\documentclass[reqno,oneside,10pt]{amsart}

\usepackage{color}

\usepackage{hyperref}
\usepackage{graphicx}
\usepackage{enumerate}
\usepackage[all]{xy}
\usepackage{amsmath,amsfonts,amssymb}
\usepackage[latin9]{inputenc}

\newcommand{\Z}{{\mathbb Z}}

\newtheorem{prop}{Proposition}[section]
\newtheorem{cor}[prop]{Corollary} 
\newtheorem{teo}[prop]{Theorem}
\newtheorem{thm}[prop]{Theorem}
\newtheorem{lemma}[prop]{Lemma}

\theoremstyle{definition}
\newtheorem{defi}[prop]{Definition}
\newtheorem{exmp}[prop]{Example}

\theoremstyle{remark}
\newtheorem{remark}[prop]{Remark}

\newcommand{\vai}{\rightarrow}

\newcommand{\ud}{{\underline{\Delta}}}
\newcommand{\ue}{{\underline{\epsilon}}}

\newcommand{\ve}{\varepsilon}
\newcommand{\te}{\tilde{\varepsilon}}

\newcommand{\um }{1_A}

\newcommand{\benu}{\begin{enumerate}}
\newcommand{\enu}{\end{enumerate}}

\newcommand{\beqna}{\begin{eqnarray}}
\newcommand{\eqna}{\end{eqnarray}}
\newcommand{\beqnast}{\begin{eqnarray*}}
\newcommand{\eqnast}{\end{eqnarray*}}
\newcommand{\beqn}{\begin{equation}}
\newcommand{\eqn}{\end{equation}}
\newcommand{\beqnst}{\begin{equation*}}
\newcommand{\eqnst}{\end{equation*}}

\usepackage{latexsym}

\usepackage{xspace}
\usepackage{amscd}
\usepackage{amssymb}
\usepackage{amsfonts}

\makeatother

\newcommand{\bema}{\left ( \begin{array}}
\newcommand{\ema}{\end{array} \right )}

\newcommand{\End}{\operatorname{End}}

\newcommand{\ot}{\otimes}

\newcommand{\anti}{S^{-1}}

\newcommand{\q}{q}

  \def\ev{{\rm ev}} 

\def\Alg{{\sf Alg}}
\def\ParRep{{\sf ParRep}}
\def\ParAct{{\sf ParAct}}
\def\ul{\underline}
\newcommand{\thlabel}[1]{\label{th:#1}}
\newcommand{\thref}[1]{Theorem~\ref{th:#1}}

\newcommand{\prlabel}[1]{\label{pr:#1}}
\newcommand{\prref}[1]{Proposition~\ref{pr:#1}}

\newcommand{\relabel}[1]{\label{re:#1}}
\newcommand{\reref}[1]{Remark~\ref{re:#1}}
\newcommand{\exlabel}[1]{\label{ex:#1}}
\newcommand{\exref}[1]{Example~\ref{ex:#1}}
\newcommand{\delabel}[1]{\label{de:#1}}
\newcommand{\deref}[1]{Definition~\ref{de:#1}}
\newcommand{\eqlabel}[1]{\label{eq:#1}}
\newcommand{\equref}[1]{(\ref{eq:#1})}
\def\Cc{{\mathcal C}}
\newcommand{\repref}[1]{PR\ref{#1}}
\newcommand{\actref}[1]{(PA\ref{#1})}

\def\Mm{{\mathcal M}}
\def\bul{\bullet}

\def\acts{\blacktriangleright}

\begin{document}

\title[Partial Representations of Hopf Algebras]{Partial Representations of Hopf Algebras}

\author[M.M.S.\ Alves]{Marcelo \ Muniz \ S. \ Alves}
\address{Departamento de Matem\'atica, Universidade Federal do Paran\'a, Brazil}
\email{marcelo@mat.ufpr.br}
\author[E.\ Batista]{Eliezer Batista}
\address{Departamento de Matem\'atica, Universidade Federal de Santa Catarina, Brazil}
\email{ebatista@mtm.ufsc.br}
\author[J.\ Vercruysse]{Joost Vercruysse}
\address{D\'epartement de Math\'ematiques, Universit\'e Libre de Bruxelles, Belgium}
\email{jvercruy@ulb.ac.be}
\thanks{{\bf Key
words and phrases:} partial Hopf action, partial action, partial smash product, partial representation, Hopf algebroid,
monoidal category. } 

\begin{abstract}
In this work, the notion of partial representation of a Hopf algebra is introduced and its relationship with partial
actions of Hopf algebras is explored. Given a Hopf algebra $H$, one can associate it to a Hopf algebroid $H_{par}$ which
has the universal property  that each partial representation of $H$ can be factorized by an algebra morphism from
$H_{par}$. We define also the category of partial modules over a Hopf algebra $H$, which is the category of modules over
its associated Hopf algebroid $H_{par}$. The Hopf algebroid structure of $H_{par}$ enables us to enhance the category of
partial $H$ modules with a monoidal structure and such that the algebra objects in this category are the usual partial
actions.  Some examples of categories of partial $H$ modules are explored. In particular we can describe fully the
category of partially 
$\mathbb{Z}_2$-graded modules.
 \end{abstract} 

\maketitle

\tableofcontents

\section{Introduction}

The concept of a partial  group action, introduced in \cite{E1,E2}, resulted from the effort to endow important 
classes of $C^*$-algebras generated by partial isometries with a structure of a more general crossed product. The new
structure permitted to obtain relevant  results on $K$-theory, ideal structure and representations  of the algebras
under consideration, as well as to treat amenability questions, especially amenability of $C^*$-algebraic bundles (also
called Fell bundles),  using both partial actions and the related concept of partial representations of groups \cite{D}.
The algebraic study of partial actions and partial representations was initiated in \cite{DE} and \cite{DEP}, motivating
investigations in diverse directions.  In particular, the Galois theory of partial group actions developed in \cite{DFP}
inspired  further Galois theoretic results in \cite{CdG}, as well as the  introduction and study of partial Hopf actions
and coactions in \cite{CJ}. The latter paper became the starting point for further investigation of partial Hopf
(co)actions in \cite{AB}, \cite{AB2} and \cite{AB3}. One of the main results related to partial actions of Hopf algebras
is the globalization theorem, which states that every partial action of a Hopf algebra on a unital algebra can be viewed
as a restriction of a global action to a unital ideal.

The concept of a partial representation of a Hopf algebra first appeared in \cite{AB}, but there, only partial actions
on right ideals were considered, therefore, partial representations were presented in an asymmetric way. Basically,
a partial representation provides a way to still control the multiplicative behaviour of a linear map that fails to
be an algebra morphism. 
More specifically, a linear map $\pi$ from a Hopf algebra $H$ into an algebra $B$ is a partial representation if it
satisfies the axioms of the definition \ref{partialrep}, this set of conditions tells that it is only possible to join
the product in the presence of a ``witness''. For the case when the Hopf algebra is a group algebra, the conditions for
a partial representation coalesce into the conditions for a partial representation of a group \cite{DEP}. Partial
representations of groups were explored in references \cite{DEP} and \cite{DZ} and they are strictly related to
groupoids. In the case of partial representations of Hopf algebras, as we shall see in this article, one finds a richer
and more complex structure.

It is well known that the theory of partial actions of Hopf algebras has a lot of similarities with the theory of weak
Hopf algebras. The techniques used to develop partial Galois theory \cite{CdG} or partial entwining structures
\cite{CJ}, has a lot in common  with the techniques used for weak Hopf-Galois theory and weak entwining structures. This
was clarified in \cite{CJ}, showing that  both weak an partial entwining structures are instances of a more general,
so called ``lax'' structure. From a slightly different perspective similar connections were investigated in a more
categorical setting in \cite{Bohm:wtm}, where it was shown that both weak and partial entwining structures provide
examples of generalized liftings of monads in bicategories. 
Despite these similarities, the ``weak'' Hopf world has been explored a lot more than the
``partial'' Hopf world. 
So far, in the ``partial world'' only {\em partial actions} of a (usual) Hopf algebra $H$ on an algebra $A$ were defined
\cite{CJ}. Somewhat suggestively, such an algebra $A$ was then called a partial $H$-module algebra, although there did
not exists such a thing as a partial $H$-module. In fact, in classical Hopf algebra theory as well as in the weak case,
module algebras are exactly algebra objects in a monoidal category of modules. More precisely, we know that the category
of modules over a weak Hopf $k$-algebra is a monoidal category that allows a separable-Frobenius forgetful functor to
the category of $k$-modules. In fact, the category of these modules is equivalent to a category of modules over a Hopf
algebroid that is constructed out of the weak Hopf algebra. The base algebra of this Hopf algebroid is a
seperable-Frobenius algebra. 
Using partial representations, and the classical analogy between representations and modules, we
introduce in this paper the notion of a partial $H$-module $M$, as being a $k$-module allowing a partial representation
of $H$ on $\End(M)$. We can then show that the category of partial representations is isomorphic to the category of
modules over a certain algebra $H_{par}$ that is associated to $H$. Moreover, this algebra $H_{par}$ turns out to be a
Hopf algebroid, which shows once more the analogy between the weak and partial setting. The main difference between the
Hopf algebroids that appear in the partial setting compared to those associated to a weak Hopf algebra is that
the base algebra of the Hopf algebroid $H_{par}$ is not necessarily separable-Frobenius. 
As a consequence of the description via the Hopf algebroid, the category of partial modules is a monoidal
category. Algebra objects in this category turn out to be exactly (symmetric) partial $H$-module algebras.

Moreover, the Hopf algebroid $H_{par}$ has 
the following universal property: For each partial representation $\pi$, of $H$ on an algebra
$B$, there exists a unique algebra morphism $\hat{\pi}$, from $H_{par}$ to $B$ which factorizes this partial
representation. The association of $H_{par}$ to a Hopf algebra $H$ is in fact a functor from the category of Hopf
algebras to the category of Hopf algebroids. For the group case, this universal algebra was already know to be the
algebra of a groupoid \cite{DEP}. In case of a finite group $G$, the associated groupoid is again finite. This raises
the question whether the universal Hopf algebroid $H_{par}$ associated to a finite dimensional Hopf algebra is always
finite dimensional. We show in this paper that this is not the case, as the Hopf algebroid associated to Sweedler's
four dimensional Hopf algebra turns out to be infinite dimensional. 

We also explore a larger class of examples that go beyond the case of partial actions of groups, being the partially
graded algebras. Indeed, recall that given a group $G$, (classically) $G$-graded $k$-modules are exactly the comodules
over the group algebra $kG$, and $G$-graded algebras are the algebra objects in this monoidal category. For a finite
group $G$, the comodules over $kG$ are nothing else than the modules over the dual algebra $kG^*$. Hence we define
partially $G$-graded modules and and partially $G$-graded algebras as partial modules and partial actions of
$kG^*$. 
There is no general
method yet to explore all such partially graded objects, but we can show how to construct partially $G$ graded algebras
out of partial gradings on the base field. Also, given a normal subgroup $H\trianglelefteq G$, one can lift $G/H$
gradings to partial $G$ gradings in a canonical way. Finally, we describe completely the example of the category of
partially $\mathbb{Z}_2$ graded modules and give a characterization of partially $\mathbb{Z}_2$ graded algebras.

This paper is organized as follows. In section 2, the theory of partial representations of groups is reviewed. The
results related to partial representations of groups can be found in references \cite{DE} and \cite{DEP}, therefore, we
omit the proofs of the most part of the results in that section, except for the proof of the isomorphism between the
partial group algebra and a partial skew group ring, which contains many important ideas used for the case of partial
representations of Hopf algebras. Also in that section, we present the universal property of the partial group algebra
in a more categorical language. In section 3, we introduce partial representations of Hopf algebras and give the
paradigmatic examples of them, namely, the partial representation defined from a partial action and the partial
representation related to the partial smash product. In section 4, we introduce the universal
algebra which factorizes every partial action of a Hopf algebra $H$ by an algebra morphism, we call it the universal
partial ``Hopf'' algebra $H_{par}$ 
We show that the universal partial ``Hopf'' algebra $H_{par}$
has the structure of a Hopf algebroid and it is isomorphic to a partial smash product. Using these general results, we
can calculate thoroughly the example of the partial ``Hopf'' algebra for the Sweedler Hopf algebra $H_4$. In section 5,
we define the monoidal category of partial modules over a Hopf algebra $H$, as the category of modules over $H_{par}$.
and show that the algebra objects in this category correspond to partial actions. 
Some examples of categories of partial $H$ modules are explored
in detail. For the case of a group algebra $H=kG$, this category coincides with the category of partial actions of the
group $G$ on $k$-vector spaces. The partial modules over a universal  enveloping algebra of a Lie algebra $\mathfrak{g}$
coincide with the usual modules. Finally, in section 6, we explore some properties of partially graded modules and
partially graded algebras and provide some explicit examples.

Throughout the paper, all algebras are modules over a commutative ring $k$ and will be considered unital and
associative, unless mentioned otherwise.

\section{Partial representations of groups}

Partial representations of groups are closely related to partial actions of groups on algebras. In fact, to each partial action of a group on an algebra one can associate a partial representation of the same group and each partial representation of a group factors trough a universal algebra which is isomorphic to a partial skew group algebra related to a particular partial action of the same group. Moreover, there are deeper connections between partial representations of groups and inverse semigroups \cite{E2}.

\begin{defi} \cite{DE}
A {\em partial representation} of a group $G$ on an algebra $B$ is a map $\pi :G\rightarrow B$   such that 
\begin{enumerate}
\item $\pi (e) = 1_B$,
\item $\pi (g)\pi (h) \pi (h^{-1}) = \pi (gh)\pi (h^{-1}), \qquad \forall g,h \in
G$
\item $\pi (g^{-1}) \pi (g) \pi (h) = \pi (g^{-1}) \pi (gh),\qquad \forall g,h \in
G.$
\end{enumerate} 

Let $\pi:G\to B$ and $\pi':G\to B'$ be two partial representations of $G$. A {\em morphism of partial partial representations} is an
algebra morphisms $f:B\rightarrow B'$ such that $\pi ' =f\circ \pi$.

The category of partial representations of $G$, denoted as $\ParRep_G$ is the category whose objects are pairs $(B,\pi)$, where $B$ is a unital $k$-algebra and $\pi : G\rightarrow B$ is a partial representation of $G$ on $B$, and whose morphisms are morphisms of partial representations.
\end{defi}

\begin{remark}
Evidently, representations of $G$ are automatically partial representations. And if one put an extra condition
\[
\pi (g) \pi (g^{-1}) =1_B , \quad \forall g\in G ,
\]
then the partial representation $\pi$ turns out to be a usual representation of the group $G$ on the algebra $B$. Indeed
\[
\pi (g) \pi (h) = \pi (g) \pi (h) \pi (h^{-1}) \pi (h) =\pi (gh) \pi (h^{-1}) \pi (h) =\pi (gh) .
\]
\end{remark}

To see how partial representations of groups are closely related to the concept of partial actions of groups, let us briefly remember some facts about partial group actions.

\begin{defi} \cite{DE} 
 Let $G$ be a group and $A$ an algebra, a {\em partial action} $\alpha$ of $G$ on $A$ is given by a collection $\{D_g \}_{g\in G}$ of ideals of $A$ and a collection 
$ \{\alpha_g :D_{g^{-1}}\rightarrow D_g \}_{g\in G}$ of (not necessarily unital) algebra isomorphisms, satisfying the following conditions: 
\begin{enumerate}
\item $D_e =A$, and $\alpha_e =\mbox{Id}_A$.
\item $\alpha_g (D_{g^{-1}} \cap D_h )=D_g \cap D_{gh}$.
\item If $x\in D_{h^{-1}}\cap D_{(gh)^{-1}}$, then $\alpha_g \circ \alpha_h =\alpha_{gh}$.
\end{enumerate}
Consider two partial action $(A,\{D_g\}_{g\in G},\{\alpha_g\}_{g\in G})$ and $(B,\{E_g\}_{g\in G},\{\beta_g\}_{g\in G})$. A morphism of partial actions is an algebra morphism $\phi:A\to B$ such that $\phi(D_g)\subseteq E_g$ and $\phi\circ \alpha_g=\beta_g\circ\phi$ for all $g\in G$.

Partial actions and the morphisms between them form a category that we denote as $\ParAct_G$.
\end{defi}

\begin{remark}
Henceforth, for sake of simplicity, we will consider a special case of partial action in which every ideal $D_g$ is of the form $D_g =1_g A$, where $1_g$ is a central idempotent of $A$ for each $g\in G$. In this case, each $D_g$ is a unital algebra and whose unit is given by $1_g$, in particular $1_e =1_A$. Moreover, each $\alpha_g$ is assumed to be an isomorphism of unital algebras, i.e.\ $\alpha_g (1_{g^{-1}} )=1_g$. Similarly, if $\phi:(A,\{1_gA=D_g\}_{g\in G},\{\alpha_g\}_{g\in G})\to (B,\{1'_gB=E_g\}_{g\in G},\{\beta_g\}_{g\in G})$ is a morphism of partial actions, then $\phi$ is supposed to satisfy $\phi(1_g)=1'_g$. 
\end{remark}

A partial action of a group $G$ on a algebra $A$ enables us to construct a new algebra, called the partial skew group algebra, denoted by $A\rtimes_{\alpha} G$. Basically
\[
A\rtimes_{\alpha} G \cong \bigoplus_{g\in G} D_g \delta_g ,
\]
as a $k$-module and with the product defined as 
\[
(a_g \delta_g )(b_h \delta_h )=\alpha_g (\alpha_{g^{-1}}(a_g ) b_h )\delta_{gh} .
\]
The associativity of the partial skew group algebra in general is discussed in 
\cite{DE}. In the particular case of partial actions in which the domains $D_g$ are ideals of the form $1_g A$, then the skew group algebra is automatically associative. It is easy to verify that the map $\pi_0 :  G  \rightarrow  A\rtimes_{\alpha} G$, given by 
$\pi_0 (g)=1_g \delta_g$ defines a partial representation of $G$.

The partial crossed product has an important universal property. Let $A$ be an algebra on which the group $G$ acts partially, define the canonical inclusion
\[
\begin{array}{rccc} \phi_0 :& A & \rightarrow & A\rtimes_{\alpha} G\\
\, & a & \mapsto & a\delta_e \end{array}
\]
which is easily seen to be an algebra monomorphism. Given a unital algebra $B$, a pair of maps $(\phi , \pi)$ is said to be a covariant pair if $\phi :A \rightarrow B$ is an algebra morphism and $\pi :G \rightarrow B$ is a partial representation such that
\[
\phi (\alpha_g (a1_{g^{-1}}))=\pi (g)\phi (a) \pi (g^{-1}) .
\]
The universal property of $A\rtimes_{\alpha} G$ is given by the following result.

\begin{thm}\thlabel{universal} \cite{DEP} Let $A$ be an algebra on which the group $G$ acts partially, $B$ a unital algebra and $(\phi ,\pi )$ a covariant pair related to these data. Then, there exists a unique algebra morphism $\Phi : A\rtimes_{\alpha} G\rightarrow B$ such that $\phi =\Phi \circ \phi_0$ and $\pi =\Phi \circ \pi_0$.
\end{thm}

There is a second partial representation we can associate to a partial action. Indeed, given a partial action $\alpha$ of $G$ on $A$, one can define a partial representation
\[
\begin{array}{rccc}
\pi : & G & \rightarrow & \mbox{End}_k (A) \\
\,    & g & \mapsto     & \pi (g)
\end{array}
\]
where $\pi (g)(a)=\alpha_g (a1_{g^{-1}})$. Moreover, we have the following.

\begin{prop}
Let $\alpha$ be a partial action of $G$ on $A$. Let $\pi:G\to \End_k(A)$ be the associated partial representation as above. Then $(\phi,\pi)$ is a covariant pair, where
\[
\phi:A\to \End_k(A),\ \phi(a)(a')=aa'.
\]
Consequently, there is an algebra morphism 
\[
\Phi_0 : A\rtimes_{\alpha} G\rightarrow \End_k (A).
\]
\end{prop}
\begin{proof} Let us verify that $(\phi,\pi)$ is covariant. Take $a,a'\in A$, then we have
\begin{eqnarray*}
\phi(\alpha_g(a1_{g^{-1}}))(a')&=&\alpha(a1_{g^{-1}})a'\\
=\pi(g)\phi(a)\pi(g^-1)(a')&=&\alpha_g(a\alpha_{g^{-1}}(a'1_g)1_{g^{-1}})\\
&=&\alpha_g(a1_{g^{-1}})\alpha_g\circ\alpha_{g^-1}(a'1_g)\\
&=&\alpha_g(a1_g^{-1})a'
\end{eqnarray*}
By \thref{universal}, we then know that the morphism $\Phi_0$ exists, it has the following explicit form
\[
\Phi_0(a_g\delta_g)(a')=a_g\alpha_g(a'1_{g^{-1}}).
\]
\end{proof}

Furthermore, the above constructions can be easily verified to be functorial. Hence, the relations between partial
actions and partial representations can be summarized in the following theorem. 

\begin{thm}
Let $G$ be a group, then there exist functors 
\begin{eqnarray*}
&\xymatrix{\ParAct_G \ar@<.5ex>[rr]^-{\Pi_0} \ar@<-.5ex>[rr]_-{\Pi} && \ParRep_G }\\ 
&\qquad \Pi_0(A,\alpha)=(A\rtimes_{\alpha} G,\pi_0), \quad \Pi(A,\alpha)=(\End_k(A),\pi)
\end{eqnarray*}
and a natural transformation $\ul\Phi:\Pi_0\to\Pi$
\end{thm}

\begin{remark}\relabel{freefunctor}
A natural question that arises from the previous theorem, is whether any of the functors $\Pi$ or $\Pi_0$ has an adjoint. We will solve this question in the more general setting of partial actions and representations of Hopf algebras, where we will observe that $\Pi$ is nothing else then a forgetful functor, hence allowing a left adjoint, associating a ``free partial action'' to a partial representation.
\end{remark}

It is well known in classical representation theory of groups that any representation $\rho$ of a group $G$ on a $k$ module $M$ corresponds to an algebra morphism between the group algebra $kG$ and the algebra $\mbox{End}_k (M)$. For the case of partial representations of a group $G$, one can define a new object associated to $G$ in order to obtain an analogous result. The ``partial group algebra'' $k_{par}G$ is the unital algebra generated by symbols $[g]$ for each $g\in G$ satisfying the relations
\[
[e]=1 \; , \quad [g][h][h^{-1}]=[gh][h^{-1}]\; , \quad [g^{-1}][g][h]=[g^{-1}][gh]\qquad \forall g,h\in G .
\]
It is easy to see that the map 
\[
\begin{array}{rccc}[ \underline{\, }] : & G & \rightarrow & k_{par}G \\
\, & g & \mapsto & [g] \end{array}
\]
is a partial representation of $G$ on the algebra $k_{par}G$. The next proposition characterizes $k_{par}G$ by a universal property. Before we formulate this proposition, recall that given a category $\Cc$ with fixed object $A$, the co-slice category $A/\Cc$ is the category whose objects are pairs $(B,f)$, where $B$ is an object of $\Cc$ and $f:A\to B$ is a morphism of $\Cc$. A morphism $(B,f)\to (B',f')$ in $A/\Cc$ is a map $g:B\to B'$ satisfying $g\circ f=f'$.

\begin{prop} \cite{DEP} Each partial representation $\pi :G \rightarrow B$ is associated to a unique algebra morphism $\hat{\pi} :k_{par}G\rightarrow B$, such that $\pi =\hat{\pi}\circ [\underline{\, }]$. Conversely, given any algebra morphism 
$\phi :k_{par}G\rightarrow B$, one can define a partial representation $\pi_{\phi} :G\rightarrow B$ such that $\phi =\hat{\pi_{\phi}}$. 

In other words, the following functors establish an isomorphism between the category of partial representations and the co-slice category $k_{par}G/ \Alg_k$ 
\[
\xymatrix{
\ParRep_G \ar@<.5ex>[rr]^-{L} && k_{par}G / \Alg_k \ar@<.5ex>[ll]^-{R}
}
\]
where $L((B, \pi ))= (B,\hat{\pi})$, and $R((B, \phi )) =(B,\pi_{\phi})$.
\end{prop}

\begin{proof} Define, for any word $w=[g_1][g_2]\ldots [g_n]\in k_{par}G$ the element 
$\hat{\pi}(w)=\pi (g_1 )\pi (g_2)\ldots \pi (g_n) \in B$ it is easy to see that the map $\hat{\pi}$ defined this way is the unique algebra homomorphism between $k_{par}G$ and $B$ which factorizes the partial representation $\pi$.

On the other hand, given the algebra morphism $\phi :k_{par}G\rightarrow B$, define $\pi_{\phi} :G\rightarrow B$ by $\pi_{\phi} (g)=\phi ([g])$. It is easy to see that this defines a partial representation of $G$ which is factorized by $\phi$. 
\end{proof}

A very important result in the theory of partial representations of groups is that the partial group algebra is always isomorphic to a partial crossed product. First, it is important to note that the partial group algebra $k_{par}G$ has a natural $G$ grading. Indeed we can decompose, as a vector space, the whole partial group algebra as
\[
k_{par}G \cong \bigoplus_{g\in G} A_g ,
\]
where each subspace $A_g$ is generated by elements of the form $[h_1 ][h_2 ]\ldots [h_n ]$ such that $g=h_1 h_2 \ldots h_n$, and it is easy to see that the product in $k_{par}G$ makes $A_g A_h \subseteq A_{gh}$. Now, for each $g\in G$ define the element $\varepsilon_g =[g][g^{-1}]\in k_{par}G$. One can prove easily that these $\varepsilon_g$ are idempotent for each $g\in G$. These elements satisfy the following commutation relation:
\[
[g]\varepsilon_h =\varepsilon_{gh}[g] .
\]
Indeed,
\begin{eqnarray}
[g]\varepsilon_h &=& [g][h][h^{-1}] =[gh][h^{-1}]\nonumber\\
&=& [gh][(gh)^{-1}][gh][h^{-1}]=[gh][(gh)^{-1}][g]\nonumber\\
&=& \varepsilon_{gh}[g] .\nonumber
\end{eqnarray}
From this, one can prove that all $\varepsilon_g$ commute among themselves. Define the subalgebra $A=\langle \varepsilon_g |g\in G \rangle \subseteq k_{par}G$. This is a commutative algebra generated by central idempotents, and it is not difficult to prove that the sub algebra $A$ corresponds to the uniform sub algebra $A_e$ coming from the natural $G$ grading above presented. Then, we have the following result.

\begin{thm} \cite{DE} Given a group $G$, there is a partial action of $G$ on the commutative algebra $A\subseteq k_{par}G$ above defined, such that $k_{par}G \cong A\rtimes_{\alpha}G$.
\end{thm}

\begin{proof} In order to define a partial action of $G$ on $A$, we have to give the domains $D_g$ and the isomorphisms $\alpha_g :D_{g^{-1}}\rightarrow D_g$ for each 
$g\in G$. As the elements $\varepsilon_g$ are central idempotents in $A$, define the ideals $D_g =\varepsilon_g A$. Clearly, these ideals are unital algebras with unit $\varepsilon_g$. Now, the partially defined isomorphisms between these ideals are
\[
\alpha_g (\varepsilon_{g^{-1}} \varepsilon_{h_1}\ldots \varepsilon_{h_n}) 
=[g]\varepsilon_{g^{-1}} \varepsilon_{h_1}\ldots \varepsilon_{h_n} [g^{-1}] =
\varepsilon_{g} \varepsilon_{g h_1}\ldots \varepsilon_{g h_n}
\]
It is easy to verify that these data indeed define a partial action of $G$ on $A$.

In order to prove the isomorphism, let us use both universal properties, of the partial crossed product and of the partial group algebra. First, the map $\pi_0 :G\rightarrow A\rtimes_{\alpha}G$ given by $\pi_0 (g) =\varepsilon_g \delta_g$ is a partial representation of the group $G$ on the partial crossed product. Then, there is a unique algebra morphism $\hat{\pi}: k_{par}G\rightarrow A\rtimes_{\alpha}G$, which factorizes this partial representation. This morphism can be written explicitly as
\[
\hat{\pi} ([g_1]\ldots [g_n ])=\varepsilon_{g_1}\varepsilon_{g_1 g_2}\ldots 
\varepsilon_{g_1 \ldots g_n} \delta_{g_1 \ldots g_n } .
\] 

On the other hand, the canonical inclusion of $A$ into $k_{par}G$ and the canonical partial representation $g\mapsto [g]$ form a covariant pair relative to the algebra $k_{par} G$ then there is a unique algebra morphism $\Phi :A\rtimes_{\alpha}G\rightarrow k_{par}G$, explicitly given by 
\[
\Phi (\varepsilon_g \varepsilon_{h_1} \ldots \varepsilon_{h_n} \delta_g) =\varepsilon_g \varepsilon_{h_1} \ldots \varepsilon_{h_n} [g].
\]
Easily, one can verify that the morphisms $\hat{\pi}$ and $\Phi$ are mutually inverses, completing the proof.
\end{proof}

\section{Partial representations of Hopf algebras}

A first definition of partial representations of Hopf algebras was proposed in \cite{AB}, requiring only axioms
(\repref{partialrep1}) and (\repref{partialrep2}) below. This was mainly motivated by the constructions done in
\cite{CJ} for partial actions of a Hopf algebra $H$ on a unital algebra $A$, originated from partial entwining
structures. Nevertheless, many of the important results on partial actions, such as globalization theorems and Morita
equivalence between partial smash products demanded a more restrictive definition of partial action, which always fit in
the case of partial group actions. In fact, in case of a cocommutative Hopf algebra (such as the group algebra), this
definition becomes perfectly left-right symmetric considering only axioms (\repref{partialrep1}, \repref{partialrep2},
\repref{partialrep5}).

\begin{defi} \label{partialrep}
Let $H$ be a Hopf $k$-algebra, and let $B$ be a unital $k$-algebra. A \emph{partial representation} of $H$ on $B$ is a linear map $\pi: H \rightarrow B$ such that 
\begin{enumerate}[{(PR1)}]
\item $\pi (1_H)  =  1_B$ \label{partialrep1}
\item $\pi (h) \pi (k_{(1)}) \pi (S(k_{(2)}))  =   \pi (hk_{(1)}) \pi (S(k_{(2)})) $
\label{partialrep2}
\item $\pi (h_{(1)}) \pi (S(h_{(2)})) \pi (k)  =   \pi (h_{(1)}) \pi (S(h_{(2)})k)$ \label{partialrep3} 
\item $\pi (h) \pi (S(k_{(1)})) \pi (k_{(2)}) = \pi (hS(k_{(1)})) \pi (k_{(2)})$ \label{partialrep4}
\item $\pi (S(h_{(1)}))\pi (h_{(2)}) \pi (k) = \pi (S(h_{(1)}))\pi (h_{(2)} k)$ \label{partialrep5}
\end{enumerate}

If $(B,\pi)$ and $(B',\pi')$ are two partial representations of $H$, then we say that an algebra morphism $f:B\rightarrow B'$ is a morphism of partial representations if $\pi ' =f\circ \pi$. 

The category whose objects are partial representations of $H$ and whose morphisms are morphisms of partial representations is denoted by $\ParRep_H$.
\end{defi}

\begin{remark} If the Hopf algebra $H$ has invertible antipode then the items 
(\repref{partialrep3}) and (\repref{partialrep4}) can be rewritten respectively as
\begin{enumerate}
\item[(PR3)] $\pi (\anti (h_{(2)})) \pi (h_{(1)}) \pi (k)  =   \pi (\anti (h_{(2)})) \pi (h_{(1)}k)$.
\item[(PR4)] $\pi (h) \pi (k_{(2)}) \pi (\anti (k_{(1)})) = \pi (hk_{(2)}) \pi (\anti (k_{(1)}))$
\end{enumerate}

If the Hopf algebra $H$ is cocommutative, then the items in the definition of a partial representation coalesce into (\repref{partialrep1}), (\repref{partialrep2}) and (\repref{partialrep5}).
\end{remark}

If $\pi :H \rightarrow B$ is a representation of $H$ on $B$ (i.e.\ $\pi$ is an algebra morphism), then it is automatically a partial representation. If the partial representation $\pi$ satisfies the extra condition
\[
\pi (S(h_{(1)}))\pi (h_{(2)}) =\epsilon (h)1_B ,\quad \forall h\in H ,
\]
one can prove that $\pi$ is an algebra morphism of $H$ on $B$. Indeed,
\beqnast
\pi (h) \pi (k) & = & \pi (h) \pi (k_{(1)}) \epsilon (k_{(2)}) \\
& = &  \pi (h) \pi (k_{(1)}) \pi (S(k_{(2)}))\pi (k_{(3)}) \\
& = &  \pi (hk_{(1)}) \pi (S(k_{(2)}))\pi (k_{(3)}) \\
& = &  \pi (hk_{(1)}) \epsilon (k_{(2)})\\
& = & \pi (hk) .
\eqnast
The same is true if the condition 
\[
\pi (h_{(1)})\pi (S(h_{(2)})) =\epsilon (h) 1_B, \quad \forall h\in H,
\]
holds and the proof can be carried out in the same way using the item (\repref{partialrep5}).

\begin{prop} 
Given any partial representation $\pi :H\rightarrow B$, for all $h,k \in H$, we have the following equalities 
\beqna
\pi (S(h_{(1)})) \pi (h_{(2)}) \pi (S(h_{(3)})) & = & \pi (S(h)) \label{key1} \\
\pi (h_{(1)})\pi (S(h_{(2)})) \pi (h_{(3)}) & = & \pi (h) \label{key2}
\eqna
If the Hopf algebra $H$ has invertible antipode, then we have the extra equalities
\beqna
\pi (h_{(3)})\pi (\anti (h_{(2)})) \pi (h_{(1)}) & = & \pi (h) \label{key3}\\
\pi (\anti (h_{(3)})) \pi (h_{(2)}) \pi (\anti (h_{(1)})) &=&\pi (\anti (h)) \label{key4}
\eqna
\end{prop}

\begin{proof} The first equality comes from the relations (\repref{partialrep1}) and (\repref{partialrep2}), the second is easily deduced from (\repref{partialrep1}) and (\repref{partialrep3}). If the Hopf algebra has invertible antipode, then (\ref{key3}) follows from (\ref{key1}) and (\ref{key4}) from (\ref{key2}) taking $h=\anti (h' )$.
\end{proof}

Partial representations of Hopf algebras are closely related partial actions of Hopf algebras.

\begin{defi} \cite{CJ} A left partial action of a Hopf algebra $H$ on a unital algebra $A$ is a linear map 
\[
\begin{array}{rccc} \cdot : & H \otimes A & \rightarrow & A \\
\, & h\otimes a & \mapsto & h\cdot a \end{array}
\]
such that
\begin{enumerate}[({PA}1)]
\item $1_H \cdot a =a$ for all $a\in A$.\label{partact1}
\item $h\cdot (ab)=(h_{(1)}\cdot a)(h_{(2)}\cdot b)$, $\forall \; h\in H$, $a,b\in A$.\label{partact2}
\item $h\cdot (k\cdot a)=(h_{(1)}\cdot \um )(h_{(2)}k\cdot a)$, $\forall \; h,k\in H$, $a\in A$.\label{partact3}
\end{enumerate}
The algebra $A$, on which $H$ acts partially is called a partial left $H$ module algebra.

We say that a partial action of $H$ on $A$ is {\em symmetric} if, in addition it satisfies 
\begin{enumerate}[({PA}1)]
\setcounter{enumi}{3}
\item $h\cdot (k\cdot a)=(h_{(1)}k\cdot a)(h_{(2)}\cdot \um ), \forall \; h,k\in H, a\in A.$ \label{partact3b}
\end{enumerate} 

If $(A,\cdot)$ and $(B,\cdot)$ are two partial left $H$-module algebras, then a morphism of partial $H$-module algebras
is an algebra map $f:A\to B$ such that $f(h\cdot a)=h\cdot f(b)$. The category of all symmetric partial left $H$-module
algebras and the morphisms of partial $H$-module algebras between them is denoted as $\ParAct_H$.
\end{defi}

Given a symmetric partial action of a Hopf algebra $H$, one can define a partial representation of $H$, as shown in the following example.

\begin{exmp}\label{example.partial.rep.in.End(A)}\exlabel{partact}
Let $A$ be a symmetric partial left $H$-module algebra, and let $B = \End(A)$. Define $\pi : H \rightarrow B $ by
$\pi(h) (a) = h \cdot a$.  Then $1_H \cdot a=a$, for all $a$ in $A$, implies  $\pi(1_H) = I_A$, so
(\repref{partialrep1}) is satisfied. 

With respect to (\repref{partialrep2}), we have on one hand
\beqnast
\pi (h)\pi (k_{(1)}) \pi (S(k_{(2)})) (a) & = & h \cdot (k_{(1)} \cdot (S(k_{(2)}) \cdot a)) \\
& \stackrel{\actref{partact3}}{=} &  h \cdot ( (k_{(1)} \cdot \um )(k_{(2)} S(k_{(3)}) \cdot a)) \\
& = &  h \cdot ((k \cdot \um ) a) 
\eqnast
and on the other hand, we have
\beqnast
\pi (hk_{(1)})\pi (S(k_{(2)}))(a) & = & hk_{(1)} \cdot (S(k_{(2)}) \cdot a) \\
& \stackrel{\actref{partact3}}{=} &  (h_{(1)}k_{(1)} \cdot \um ) (h_{(2)}k_{(2)}S(k_{(3)}) \cdot a))\\
& = &  (h_{(1)}k \cdot \um )(h_{(2)} \cdot a)\\
& \stackrel{\actref{partact2}}{=} &  (h_{(1)}k \cdot \um )(h_{(2)} \cdot \um )(h_{(3)} \cdot a)\\
& \stackrel{\actref{partact3b}}{=} &  (h_{(1)} \cdot (k \cdot \um ))(h_{(2)} \cdot a)\\
& \stackrel{\actref{partact2}}{=} &  h \cdot ((k \cdot \um )a)
\eqnast

For equation (\repref{partialrep3}) we have 
\begin{eqnarray*}
\pi (h_{(1)}) \pi (S(h_{(2)})) \pi (k) (a) &=& h_{(1)} \cdot (S(h_{(2)}) \cdot (k\cdot a)) \\
&\stackrel{\actref{partact3}}{=}& ( h_{(1)} \cdot \um ) ( h_{(2)} S(h_{(3)}) \cdot (k\cdot a)) \\
&=& ( h \cdot \um )(k\cdot a)
\end{eqnarray*}
and 
\begin{eqnarray*}
\pi (h_{(1)}) \pi (S(h_{(2)}) k) (a) &=& h_{(1)} \cdot (S(h_{(2)}) k\cdot a) \\
&\stackrel{\actref{partact3}}{=}& ( h_{(1)} \cdot \um ) ( h_{(2)} S(h_{(3)})k\cdot a) \\
&=& ( h\cdot \um )(k\cdot a)
\end{eqnarray*}
which proves the equality.

For the equation (\repref{partialrep4}), we have
\begin{eqnarray*}
\pi (h) \pi (S(k_{(1)})) \pi (k_{(2)}) (a) & = & h\cdot ( S(k_{(1)}) \cdot (k_{(2)} \cdot a)) \\
& \stackrel{\actref{partact3b}}{=} & h\cdot (S(k_{(2)})k_{(3)}\cdot a)(S(k_{(1)})\cdot \um ) \\
& = & h\cdot (a(S(k)\cdot \um ) =\\
& \stackrel{\actref{partact2}}{=} & (h_{(1)} \cdot a)(h_{(2)}S(k)\cdot \um )
\end{eqnarray*}
and
\begin{eqnarray*}
\pi (hS(k_{(1)})) \pi (k_{(2)}) (a) & = & hS(k_{(1)}) \cdot (k_{(2)} \cdot a) \\
& \stackrel{\actref{partact3b}}{=} & (h_{(1)}S(k_{(2)})k_{(3)}\cdot a)(h_{(2)}S(k_{(1)})\cdot \um ) \\
& = & (h_{(1)} \cdot a)(h_{(2)}S(k)\cdot \um )
\end{eqnarray*}

Finally, for the equation (\repref{partialrep5}), we have
\begin{eqnarray*}
\pi (S(h_{(1)}))\pi (h_{(2)}) \pi (k) (a) & = & S(h_{(1)})\cdot (h_{(2)} \cdot (k\cdot a)) \\
& \stackrel{\actref{partact3b}}{=} & (S(h_{(2)})h_{(3)} \cdot (k\cdot a))(S(h_{(1)})\cdot \um ) \\
& = & (k\cdot a)(S(h)\cdot \um)
\end{eqnarray*}
and
\begin{eqnarray*}
\pi (S(h_{(1)}))\pi (h_{(2)} k) (a) & = & S(h_{(1)})\cdot (h_{(2)} k\cdot a) \\
& \stackrel{\actref{partact3b}}{=} & (S(h_{(2)})h_{(3)} k\cdot a)(S(h_{(1)})\cdot \um ) \\
& = & (k\cdot a)(S(h)\cdot \um )
\end{eqnarray*}

Therefore, the partial action of $H$ on $A$ provides an example of a partial representation of $H$ on $\mbox{End}_k (A)$.
\end{exmp}

Given a partial action of a Hopf algebra $H$ on a unital algebra $A$, one can define an associative product on $A\otimes H$, given by
\[
(a\otimes h)(b\otimes k)=a(h_{(1)}\cdot b)\otimes h_{(2)}k
\]
Then, a new unital algebra is constructed as
\[
\underline{A\# H}=(A\otimes H)(\um \otimes 1_H ) .
\]
This unital algebra is called partial smash product \cite{CJ}. This algebra is generated by typical elements of the form
\[
a\# h= a(h_{(1)}\cdot \um )\otimes h_{(2)} ,
\]
and throughout this article, we shall use this abbreviated notation for the elements of the partial smash product. A straightforward calculation gives us the following properties of the partial smash product.

\begin{lemma} \label{propertiesofsmash} \cite{AB} Let $H$ be a Hopf algebra acting partially on a unital algebra $A$, then the elements of the partial smash product $\underline{A\# H}$ satisfy
\begin{enumerate}[(i)]
\item $(a\# h )(b\# k)=a(h_{(1)}\cdot b)\# h_{(2)}k$.
\item $a\# h =a(h_{(1)} \cdot \um ) \# h_{(2)}$.
\item The map $\phi_0 :A\rightarrow \underline{A\# H}$ given by $\phi_0 (a)=a\# 1_H$ is an algebra morphism.
\end{enumerate}
\end{lemma}

Partial smash products give another source of examples of partial representations of a Hopf algebra $H$.

\begin{exmp}\label{smash}\exlabel{smash} Given a partial action of a Hopf algebra on a unital algebra 
$A$, the linear map $\pi_0 :H\rightarrow \underline{A\# H}$ given by $\pi_0 (h)=\um \# h$ is a partial representation of $H$. First, the item (\repref{partialrep1}) is easily obtained, for $\pi_0 (1_H )=\um \# 1_H =1_{\underline{A\# H}}$. Now, verifying the item (\repref{partialrep2}), we have on one hand
\beqnast
\pi_0 (h)\pi_0 (k_{(1)}) \pi_0 (S(k_{(2)})) & = & (\um \# h)(\um \# k_{(1)})(\um \# S(k_{(2)})) \\
& = & (\um \# h)((k_{(1)}\cdot \um ) \# k_{(2)}S(k_{(3)})) \\
& = & (\um \# h )((k\cdot \um )\# 1_H) \\
& = & (h_{(1)} \cdot (k\cdot \um ))\# h_{(2)}
\eqnast
and on the other hand
\beqnast
\pi_0 (hk_{(1)}) \pi_0 (S(k_{(2)})) & = & (\um \# hk_{(1)})(\um \# S(k_{(2)})) \\
& = & ((h_{(1)}k_{(1)}\cdot \um ) \# h_{(2)}k_{(2)}S(k_{(3)})) \\
& = & ((h_{(1)}k\cdot \um )\# h_{(2)}) \\
& = & ((h_{(1)}k\cdot \um )(h_{(2)}\cdot \um )\# h_{(3)}) \\
& = & (h_{(1)} \cdot (k\cdot \um ))\# h_{(2)}
\eqnast
For the equation (\repref{partialrep3}), we have
\begin{eqnarray*}
\pi_0 (h_{(1)})\pi_0 (S(h_{(2)}))\pi_0 (k)&=&
(1_A \# h_{(1)})(1_A \# S(h_{(2)}))(1_A \# k)\\
&=& (h_{(1)}\cdot 1_A \# h_{(2)}S(h_{(3)}))(1_A \# k)\\
&=& (h \cdot 1_H) (1_A \# k)=h\cdot 1_A \# k
\end{eqnarray*}
and
\begin{eqnarray*}
\pi_0 (h_{(1)})\pi_0 (S(h_{(2)})k)&=&
(1_A \# h_{(1)})(1_A \# S(h_{(2)})k)\\
&=& h_{(1)}\cdot 1_A\# h_{(2)}S(h_{(3)})k = h\cdot 1_A \# k
\end{eqnarray*}

The relations (\repref{partialrep4}) and (\repref{partialrep5}) are less obvious. For 
(\repref{partialrep4}), we have
\begin{eqnarray*}
\pi_0 (h)\pi_0 (S(k_{(1)}))\pi_0 (k_{(2)})&=&
(1_A \# h)(1_A \# S(k_{(1)}))(1_A \# k_{(2)})\\
& = & (\um \# h)((S(k_{(2)})\cdot \um )\# S(k_{(1)})k_{(3)}) \\
& = & (h_{(1)}\cdot (S(k_{(2)})\cdot \um ))\# h_{(2)} S(k_{(1)})k_{(3)} \\
& = & (h_{(1)}\cdot \um )(h_{(2)}S(k_{(2)})\cdot \um )\# h_{(3)} S(k_{(1)})k_{(3)}
\end{eqnarray*}
and
\begin{eqnarray*}
\pi_0 (hS(k_{(1)}))\pi_0 (k_{(2)})&=&
(1_A \# h S(k_{(1)}))(1_A \# k_{(2)})\\
& = & (h_{(1)}S(k_{(2)})\cdot \um )\# h_{(2)} S(k_{(1)})k_{(3)} \\
& = & (h_{(1)}S(k_{(3)})\cdot \um )( h_{(2)} S(k_{(2)})k_{(4)} \cdot \um ) \# h_{(3)} S(k_{(1)})k_{(5)} \\
& = & (h_{(1)}S(k_{(2)})\cdot ( k_{(3)} \cdot \um )) \# h_{(2)} S(k_{(1)})k_{(4)} \\
& = & ( h_{(2)} S(k_{(3)})k_{(4)} \cdot \um ) (h_{(2)}S(k_{(2)})\cdot \um )\# h_{(3)} S(k_{(1)})k_{(5)} \\
& = & (h_{(1)}\cdot \um )(h_{(2)}S(k_{(2)})\cdot \um )\# h_{(3)} S(k_{(1)})k_{(3)}
\end{eqnarray*}

Finally, for (\repref{partialrep5}), we have
\begin{eqnarray*}
\pi_0 (S(h_{(1)}))\pi_0 (h_{(2)}) \pi_0 (k) &=&
(1_A \# S(h_{(1)}))(1_A \# h_{(2)})(\um \# k) \\
& = & ((S(h_{(2)})\cdot 1_A ) \# S(h_{(1)})h_{(3)})(\um \# k) \\
& = & (S(h_{(3)})\cdot 1_A )(S(h_{(2)})h_{(4)} \cdot \um ) \# S(h_{(1)})h_{(5)} k \\
& = & (S(h_{(2)})\cdot (h_{(3)} \cdot \um )) \# S(h_{(1)})h_{(4)} k \\
& = & (S(h_{(3)})h_{(4)} \cdot \um ) (S(h_{(2)})\cdot 1_A )\# S(h_{(1)})h_{(5)} k \\
& = & (S(h_{(2)})\cdot 1_A )\# S(h_{(1)})h_{(3)} k \\
& = & (1_A \# S(h_{(1)}))(1_A \# h_{(2)} k) \\
& = & \pi_0 (S(h_{(1)}))\pi_0 (h_{(2)} k)
\end{eqnarray*}

Therefore, $\pi_0$ is indeed a partial representation of $H$ on the partial smash product $\underline{A\# H}$.
\end{exmp}
 
As in the group case, the constructions of \exref{partact} and \exref{smash} are functorial, and related by a natural
transformation. To obtain this natural transformation, we introduce the following definition.

\begin{defi} Consider a unital algebras $A$ and $B$, and a Hopf algebra $H$ that acts partially on $A$.
A {\em covariant pair} associated to these data is a pair of maps $(\phi ,\pi )$ where $\phi :A\rightarrow B$ is an
algebra morphism and $\pi :H\rightarrow B$ is a partial representation such that
\begin{enumerate}
\item[(CP1)] $\phi (h\cdot a)=\pi (h_{(1)}) \phi (a) \pi (S(h_{(2)}))$, for every $h\in H$ and $a\in A$ .
\item[(CP2)] $\phi (a) \pi (S(h_{(1)})) \pi (h_{(2)})=\pi (S(h_{(1)})) \pi (h_{(2)}) \phi (a)$, for every $h\in H$ and
$a\in A$.
\end{enumerate}
\end{defi}
 
This definition is well motivated because this is what happens for the pair $(\phi_0 , \pi_0)$, where $\phi_0
:A\rightarrow \underline{A\# H}$ is the canonical morphism given by $\phi_0 (a) =a\# 1_H$ and $\pi_0 :H \rightarrow
\underline{A\# H}$ is the partial representation $\pi_0 (h) =\um \# h$. Indeed,
\begin{eqnarray*}
\pi_0 (h_{(1)}) \phi_0 (a) \pi_0 (S(h_{(2)})) & = & (\um \# h_{(1)})(a\# 1_H)(\um \# S(h_{(2)})) \\
& = & (h_{(1)} \cdot a) \# h_{(2)})(\um \# S(h_{(3)})) \\
& = & (h_{(1)} \cdot a)(h_{(2)} \cdot \um )\# h_{(3)}S(h_{(3)})\\
& = & (h\cdot a)\# 1_H =\phi_0 (h\cdot a),
\end{eqnarray*}
and
\begin{eqnarray*}
\phi_0 (a) \pi_0 (S(h_{(1)})) \pi_0 (h_{(2)}) & = & (a\# 1_H )(\um \# S(h_{(1)}))(\um \# h_{(2)})\\
& = & a(S(h_{(2)})\cdot \um ) \# S(h_{(1)} h_{(3)}) ,
\end{eqnarray*}
while, on the other hand
\begin{eqnarray*}
\pi_0 (S(h_{(1)})) \pi_0 (h_{(2)})\phi_0 (a) & = & 
(\um \# S(h_{(1)}))(\um \# h_{(2)}) (a\# 1_H )\\
& = & (S(h_{(2)}) \cdot \um ) \# S(h_{(1)})h_{(2)})(a\# 1_H )\\
& = & (S(h_{(3)}) \cdot \um ) (S(h_{(2)})h_{(4)} \cdot a) \# S(h_{(1)})h_{(5)} \\
& = & (S(h_{(2)}) \cdot (h_{(3)} \cdot a)) \# S(h_{(1)})h_{(4)} \\
& = & (S(h_{(3)})h_{(4)} \cdot a) (S(h_{(2)}) \cdot \um ) \# S(h_{(1)})h_{(5)} \\
& = & a(S(h_{(2)})\cdot \um ) \# S(h_{(1)} h_{(3)}) .
\end{eqnarray*}

With this definition at hand, we can prove that the partial smash product has the following universal property.

\begin{thm} Let $A$ and $B$ be unital algebras and $H$ a Hopf algebra with a symmetric partial action on $A$. Suppose
that $(\phi ,\pi )$ is a covariant pair associated to these data. Then there exists a unique algebra morphism 
$\Phi :\underline{A\# H}\rightarrow B$ such that $\phi =\Phi \circ \phi_0$ and 
$\pi =\Phi \circ \pi_0$.
\end{thm}

\begin{proof} Define the linear map
\[
\begin{array}{rccc} \Phi :& \underline{A\# H} & \rightarrow & B \\
\, & a\# h & \mapsto & \phi (a) \pi (h) \end{array}
\]
Let us verify that this is, in fact, an algebra morphism:
\beqnast
\Phi ((a\# h)(b\# k)) & = & \Phi (a(h_{(1)}\cdot b )\# h_{(2)}k) \\
& = & \phi (a(h_{(1)}\cdot b ))\pi (h_{(2)}k)\\
& \stackrel{(CP1)}{=} & \phi (a) \pi (h_{(1)} )\phi (b) \pi (S(h_{(2)})) \pi (h_{(3)}k)\\
& \stackrel{(PR5)}{=} & \phi (a) \pi (h_{(1)} )\phi (b) \pi (S(h_{(2)})) \pi (h_{(3)}) \pi (k)\\
& \stackrel{(CP2)}{=} & \phi (a) \pi (h_{(1)} ) \pi (S(h_{(2)})) \pi (h_{(3)})\phi (b)  \pi (k)\\
& = & \phi (a) \pi (h)\phi (b) \pi (k)\\
& = & \Phi (a\# h) \Phi (b\# k) .
\eqnast
By construction, one can easily see that $\phi =\Phi \circ \phi_0$ and $\pi =\Phi \circ \pi_0$. Finally, for the
uniqueness, suppose that there is another morphism $\Psi :\underline{A\# H}\rightarrow B$ factorizing both $\phi$ and
$\pi$. Then, for any element $a\# h \in \underline{A\# H}$ we have
\beqnast
\Psi (a\# h) & = & \Psi ((a\# 1_H )(\um \# h)) =\Psi (a\# 1_H ) \Psi (\um \# h)\\
& = & \Psi (\phi_0 (a)) \Psi (\pi_0 (h))=\phi (a) \pi (h)=\Phi (a\# h).
\eqnast
\end{proof}

Thanks to the above theorem the two main examples of partial representations are related as in the following result.

\begin{thm}\thlabel{Pi}
Let $H$ be a Hopf algebra, then there exist functors 
\begin{eqnarray*}
&\xymatrix{\ParAct_H \ar@<.5ex>[rr]^-{\Pi_0} \ar@<-.5ex>[rr]_-{\Pi} && \ParRep_H }\\ 
&\qquad \Pi_0(A,\cdot )=(A\# H,\pi_0), \quad \Pi(A,\cdot)=(\End_k(A),\pi)
\end{eqnarray*}
and a natural transformation $\ul\Phi:\Pi_0\to\Pi$
\end{thm}

\begin{proof}
The construction of the functor $\Pi$ follows from \exref{partact}, the construction of the functor $\Pi_0$ follows
from \exref{smash}. It is easy to see that these constructions are indeed functorial. Moreover, if $(A,\cdot)$ is a
symmetric left parial $H$-module and $\Pi(A,\cdot)=(\End_k(A),\pi)$ the associated partial representation on $\End(A)$.
Consider $\phi:A\to \End(A), \phi(a)(a')=aa'$, then $(\phi,\pi)$ is a covariant pair. Indeed, take $h\in H$ and
$a,a'\in A$ and let us check (CP1).
\begin{eqnarray*}
\pi(h_{(1)})\phi(a)\pi(S(h_{2}))(a')&=&h_{(1)}\cdot (a(S(h_{(2)})\cdot a'))=(h_{(1)}\cdot a)(h_{(2)}\cdot
(S(h_{(3)})\cdot a))\\
&=&(h_{(1)}\cdot a)(h_{(2)}S(h_{(3)})\cdot a) = (h\cdot a)a'\\
&=& \phi(h\cdot a)(a')
\end{eqnarray*}
Next, for (CP2), we find
\begin{eqnarray*}
\pi(S(h_{(1)}))\pi(h_{(2)})\phi(a)(a')&=& S(h_{(1)})\cdot (h_{(2)}\cdot (aa')) 
= S(h_{(1)})\cdot ((h_{(2)}\cdot a)(h_{(3)}\cdot a'))\\
&=& (S(h_{(2)})\cdot (h_{(3)}\cdot a))(S(h_{(1)})\cdot (h_{(4)}\cdot a'))\\
&=& (S(h_{(2)})h_{(3)}\cdot a))(S(h_{(1)})\cdot (h_{(4)}\cdot a'))\\
&=& a(S(h_{(1)})\cdot (h_{(2)}\cdot a') = \phi(a)\pi(S(h_{(1)}))\pi(h_{(2)})(a')
\end{eqnarray*}

Consequently, there exists
an algebra morphism $\Phi:A\# H\to \End_k(A)$ such that $\pi=\Phi\circ \pi_0$, and hence $\Phi$ is a morphism of partial
representations. The verification that $\Phi$ is natural in $A$ is straightforward and left to the reader.
\end{proof}

\section{The universal partial ``Hopf'' algebra $H_{par}$}

\subsection{$H_{par}$ and the universal property}

As in the case of partial representations of a group $G$, which can be factorized by an algebra morphism defined on the
partial group algebra $k_{par}G$, in the case of partial representations of a Hopf algebra $H$, it is possible to
construct an algebra associated to $H$ which factorizes partial representations by algebra morphisms. This new algebra
will be called partial Hopf algebra.

\begin{defi} Let $H$ be a Hopf algebra 
and let $T(H)$ be the tensor algebra of the vector space $H$. The {\em partial Hopf algebra} $H_{par}$ is the quotient
of $T(H)$ by the ideal $I$ generated by elements of the form 
\begin{enumerate}
\item $1_H - 1_{T(H)}$; 
\item $h \ot k_{(1)} \ot S(k_{(2)}) - hk_{(1)} \ot S(k_{(2)})$, for all $h,k \in H$;
\item $h_{(1)} \ot S(h_{(2)}) \ot k - h_{(1)} \ot S(h_{(2)})k$, for all $h,k \in H$;
\item $h \ot S(k_{(1)}) \ot k_{(2)} - hS(k_{(1)}) \ot k_{(2)}$, for all $h,k \in H$;
\item $S(h_{(1)}) \ot h_{(2)} \ot k - S(h_{(1)}) \ot h_{(2)}k$, for all $h,k \in H$.
\end{enumerate}
\end{defi}

Denoting the class of $h\in H$ in $H_{par}$ by $[h]$, it is easy to see that the map 
\[
\begin{array}{rccc} [\underline{\; }]:& H & \rightarrow & H_{par} \\
\, & h & \mapsto & [h] \end{array}
\]
satisfy the following relations
\begin{enumerate}
\item $[\alpha h + \beta k] = \alpha [h]+ \beta [k]$, for all $\alpha, \beta \in k$ and $h,k \in H$;
\item $[1_H] = 1_{H_{par}}$; 
\item $[h][k_{(1)}][S(k_{(2)})] = [hk_{(1)}][S(k_{(2)})]$, for all $h,k \in H$;
\item $[h_{(1)}][S(h_{(2)})][k] = [h_{(1)}][S(h_{(2)})k]$, for all $h,k \in H$;
\item $[h][S(k_{(1)})][k_{(2)}] = [hS(k_{(1)})][k_{(2)}]$, for all $h,k \in H$;
\item $[S(h_{(1)})][h_{(2)}][k] = [S(h_{(1)})][h_{(2)}k]$, for all $h,k \in H$.
\end{enumerate}
Therefore, the linear map $[\underline{\;}]$ is a partial representation of the Hopf algebra $H$ on $H_{par}$.

The partial Hopf algebra $H_{par}$ has the following universal property.

\begin{thm}\thlabel{univHpar}
For every partial representation $\pi: H \rightarrow B$ there is a unique morphism of algebras $\hat{\pi}: H_{par}
\rightarrow B$ such that 
$\pi = \hat{\pi} \circ  [ \underline{\; }  ]$. Conversely, given an algebra morphism $\phi : H_{par} \rightarrow B$,
there exists a unique partial representation $\pi_{\phi} :H\rightarrow B$ such that $\phi =\hat \pi_{\phi}$.

In other words, the following functors establish an isomorphism between the category of partial representations and the
co-slice category $H_{par}/\Alg_k$
\[
\xymatrix{
\ParRep_H \ar@<.5ex>[rr]^-{L} && H_{par}/\Alg_k \ar@<.5ex>[ll]^-R
}
\]
where $L((B, \pi ))= (B,\hat{\pi})$, and $R((B, \phi )) =(B,\pi_{\phi})$.
\end{thm}

\begin{proof}
Let $(B,\pi)$ be a partial representation. From the universal property of $T(H)$ it follows that there is an algebra
morphism $\varphi: T(H) \rightarrow B$ such that $\varphi(h) = \pi(h)$ for every 
$h \in H$. Since $\pi$ is a partial representation, $\pi$ kills every generator of $I$ and, therefore, there exists a
unique algebra morphism $\hat{\pi}$ from 
$T(H)/I = H_{par}$ to $B$ such that 
\[
\hat{\pi}([h]) = \varphi (h) = \pi(h).
\]

On the other hand, given the algebra map $\phi :H_{par} \rightarrow B$, the map $\pi_{\phi}: H\rightarrow B$ defined as
$\pi_{\phi} (h)=\phi ([h])$ is a partial representation of $H$ on $B$ and is factorized by $\phi$, by a direct
verification.
\end{proof}

The identity map is evidently a partial representation of $H$ on $H$, this induces an algebra morphism $u:
H_{par}\rightarrow H$ given by $u([h^1 ]\ldots [h^n ] )=h^1 \ldots h^n$.

\begin{exmp} For the case where the Hopf algebra $H$ is a group algebra $kG$, the partial ``Hopf'' algebra $(kG)_{par}$
coincide with the well known partial group algebra 
$k_{par} G$. For a finite group $G$, this algebra has a very rich structure, being the groupoid algebra $k\Gamma (G)$
over the groupoid
\[
\Gamma (G) =\{ (A,g)\in \mathcal{P}(G) \times G \, | \, e\in A , g^{-1}\in A \} ,
\]
where $\mathcal{P}(G)$ denotes the powerset of $G$ and the product of $\Gamma(G)$ is defined as
\[
(A,g)(B,h) =\left\{ \begin{array}{lc} (B,gh), & \mbox{ if } \; A=hB \\ 
\underline{\qquad} & \mbox{ otherwise.}
\end{array} \right.
\]
The source and target maps are respectively, $s(A,g)=(A,e)$ and $t(A,g)=(gA,e)$, and the inverse is given by
$(A,g)^{-1}
=(gA,g^{-1})$. If $|G|=n$ then the dimension of $k_{par}G$ is $d=2^{n-2}(n+1)$, and moreover, this algebra can be
expressed as direct sum of matrix algebras \cite{DEP}.
\end{exmp}

\begin{exmp} If the Hopf algebra $H$ is the universal enveloping algebra of a Lie algebra $\mathfrak{g}$, that is
$H=\mathcal{U}(\mathfrak{g})$, then every partial representation of $H$ is a morphism of algebras. Indeed, consider a
partial representation $\pi :\mathcal{U}(\mathfrak{g})\rightarrow B$, let us prove that for every pair of elements
$h,k\in \mathcal{U}(\mathfrak{g})$, we have $\pi (h) \pi (k) =\pi (hk)$. This can be done by induction, considering
$k=X_1 \ldots X_n$, with $X_i \in \mathfrak{g}$, for $i=1,\ldots ,n$.

Take 
$n=1$, that is 
$k=X\in \mathfrak{g}$, then we find
\begin{eqnarray*}
\pi (h) \pi (k_{(1)})\pi (S(k_{(2)})  & = & \pi (h) \pi (X)\pi (1) -\pi (h) \pi (1)\pi (X) \nonumber \\ 
& = & \pi (h) \pi (X) -\pi (h) \pi (X) =0 
\end{eqnarray*}
On the other hand
\begin{eqnarray*}
\pi (hk_{(1)}) \pi (S(k_{(2)}) & = & \pi (hX) \pi (1) -\pi (h) \pi (X) \nonumber \\
& = & \pi (hX)-\pi (h) \pi (X) \label{step1induction2}
\end{eqnarray*}
As $\pi $ is a partial representation, we find by (\repref{partialrep2}) that the two equations above are equal, hence
\[
\pi (hX)-\pi (h) \pi (X) =0 \qquad \Rightarrow \qquad \pi (hX)=\pi (h)\pi (X) .
\]

Suppose now that for $1\leq i <n$ we have
\[
\pi (h) \pi (X_1 \ldots X_i ) =\pi (hX_1 \ldots X_i )
\]
and take $k=X_1 \ldots X_n$, then
\begin{eqnarray*}
& \,  & \pi (h) \pi (k_{(1)} \pi (S(k_{(2)}))   =  \pi (h) \pi (k) \pi (1) +(-1)^{n} \pi (h) \pi (1) \pi (k) +\\
& \, & + \sum_{i=1}^{n-1} \sum_{\sigma \in S_{n,i}} (-1)^{n-i} \pi (h) \pi (X_{\sigma (1)} \ldots X_{\sigma (i)} ) \pi
(X_{\sigma (n)} \ldots X_{\sigma (i+1)} ) \\
& = & \pi (h) \pi (k) +(-1)^{n} \pi (h) \pi (k) +\\
& \, & + \sum_{i=1}^{n-1} \sum_{\sigma \in S_{n,i}} (-1)^{n-i} \pi (hX_{\sigma (1)} \ldots X_{\sigma (i)} ) \pi
(X_{\sigma (n)} \ldots X_{\sigma (i+1)} ) ,
\end{eqnarray*}
where the sums are carried on the $(n,i)$ shuffles, for $i=1,\ldots ,n-1$. On the other hand
\begin{eqnarray*}
& \, & \pi (hk_{(1)}) \pi (S(k_{(2)} )) =  \pi (hk)\pi (1) +(-1)^{n} \pi (h) \pi (k) +\\
& \, & + \sum_{i=1}^{n-1} \sum_{\sigma \in S_{n,i}} (-1)^{n-i} \pi (hX_{\sigma (1)} \ldots X_{\sigma (i)} ) \pi
(X_{\sigma (n)} \ldots X_{\sigma (i+1)} ) .
\end{eqnarray*}
As the two equations above are equal, then we have
\[
\pi (h) \pi (k) =\pi (hk)
\]
which concludes our proof that any partial representation of the universal enveloping algebra is a morphism of algebras.
Then, the map $[\underline{\;} ]:\mathcal{U}(\mathfrak{g})\rightarrow \mathcal{U}(\mathfrak{g})_{par}$ is a morphism of
algebras and it is easy to see that it is the inverse of the algebra map $u:\mathcal{U}(\mathfrak{g})_{par} \rightarrow
\mathcal{U}(\mathfrak{g})$ defined previously. Therefore $\mathcal{U}(\mathfrak{g})_{par}\cong
\mathcal{U}(\mathfrak{g})$.
\end{exmp}

\begin{exmp} Let $k$ be a field with characteristic not  equal to 2, consider the cyclic group $C_2 = \langle e,g \, |\,
g^2 =e \rangle$, let $H$ be the dual group algebra 
$(k C_2)^*$. This Hopf algebra is generated by the basis elements $p_e$ and $p_g$ satisfying $1=p_e + p_g$ and $p_e p_g
=0$. Let $x = [p_e]$, $y = [p_g]$ be the corresponding generators of $H_{par}$. The first defining equation for
$H_{par}$
yields  $y = 1-x$, in particular, we conclude that $xy=yx$ in $H_{par}$. 

Since $H$  is cocommutative the following 4 families of equations that define $H_{par}$ are reduced to the first two of
them. Writing the equations explicitly with respect to the basis elements one obtains eight equations;  
 using the fact that $y = 1-x$, all these  equations are reduced to 
\[
x(2x-1)(x-1)  = 0 
\]
which implies that  $H_{par}$ is isomorphic to the $3$-dimensional algebra 
\[
k[x]/\langle x(2x-1)(x-1) \rangle .
\] 
\end{exmp}

\begin{remark}\relabel{Hparfinite}
As we saw in the previous examples, if $H$ is a group algebra $kG$ where $G$ is a finite group, then the associated
universal partial ``Hopf'' algebra is again finite dimensional. Also, $H_{par}$ is finite dimensional for the dual
algebra of the group algebra $kC_2$. This should be no surprise, since in this particular case, the dual group algebra
is isomorphic to the group algebra. The question whether a finite dimensional Hopf algebra has a finite dimensional
universal partial ``Hopf'' algebra, will be given a negative answer in the next section.
\end{remark}

\subsection{$H_{par}$ as a partial smash product}

As we have seen, for partial representations of groups the partial group algebra $k_{par}G$ is isomorphic to a partial
skew group ring $A\rtimes_\alpha G$.  The same kind of result can be proved in the case of partial representations of
Hopf algebras. 

Consider the partial Hopf algebra $H_{par}$ associated to $H$. For each $h \in H$, define the elements  
\[
\varepsilon_h = [h_{(1)}][S(h_{(2)})] , \qquad \te_h =[S(h_{(1)})][h_{(2)}] .
\]

\begin{lemma}Let $H$ be a Hopf algebra with invertible antipode.
\begin{enumerate}
\item[(a)] $\varepsilon_k [h] = [h_{(2)}]\varepsilon_{S^{-1}(h_{(1)})k}$;
\item[(b)] $[h]\varepsilon_k = \varepsilon_{h_{(1)}k}[h_{(2)}]$;
\item[(c)] $\varepsilon_{h_{(1)}} \varepsilon_{h_{(2)}}= \varepsilon_h$;
\item[(d)] $\te_k [h] = [h_{(1)}] \te_{kh_{(2)}}$;
\item[(e)] $[h]\te_k = \te_{kS^{-1} (h_{(2)})} [h_{(1)}]$;
\item[(f)] $\te_{h_{(1)}} \te_{h_{(2)}} =\te_h$;
\item[(g)] $\te_h \varepsilon_k = \varepsilon_k \te_h$. 
\end{enumerate}
\end{lemma}

\begin{proof} For the item (a), we have
\begin{eqnarray*}
\ve_k [h] & = & [k_{(1)}][S(k_{(2)})][h]=[k_{(1)}][S(k_{(2)})h]\\
& = &[k_{(1)}][S(k_{(2)})h_{(3)}][S^{-1}(S(k_{(3)})h_{(2)})][S(k_{(4)})h_{(1)}]\\
& = & [k_{(1)}S(k_{(2)})h_{(3)}][S^{-1}(S(k_{(3)})h_{(2)})][S(k_{(4)})h_{(1)}]\\
& = & [h_{(3)}][S^{-1}(S(k_{(1)})h_{(2)})][S(k_{(2)})h_{(1)}]\\
& = & [h_{(3)}][S^{-1}(h_{(2)})k_{(1)}][S(S^{-1}(h_{(1)})k_{(2)})]\\
& = & [h_{(2)}]\ve_{S^{-1} (h_{(1)})k}
\end{eqnarray*}

For the item (b) we have
\begin{eqnarray*}
[h]\ve_k & = & [h][k_{(1)}][S(k_{(2)})]=[hk_{(1)}][S(k_{(2)})]\\
& = & [h_{(1)}k_{(1)}][S(h_{(2)}k_{(2)})][h_{(3)}k_{(3)}][S(k_{(4)})]\\
& = & [h_{(1)}k_{(1)}][S(h_{(2)}k_{(2)})][h_{(3)}k_{(3)}S(k_{(4)})]\\
& = & [h_{(1)}k_{(1)}][S(h_{(2)}k_{(2)})][h_{(3)}]\\
& = & \ve_{h_{(1)}k}[h_{(2)}]
\end{eqnarray*}

For the item (c),
\[
\ve_h =  [h_{(1)}][S(h_{(2)})] =[h_{(1)}][S(h_{(2)})][h_{(3)}][S(h_{(4)})] =\ve_{h_{(1)}} \ve_{h_{(2)}} .
\]

The items (d) and (e) are analogous to items (a) and (b) and the item (f) is analogous to the item (c). The item (g) follow by opening $\te_h$ and applying twice the item (b).
\end{proof}

Define the subalgebras
\[
A=\langle \ve_h \in H_{par} \, | \, h\in H \rangle , \quad \mbox{ and } \quad \tilde{A}=\langle \te_h \in H_{par} \, | \, h\in H \rangle .
\]
By the previous lemma, one concludes that these two subalgebras of $H_{par}$ commute one with respect to the other. 

\begin{thm}\label{thm.isomorphism.Hpar.smashproduct}\thlabel{Hparisosmash}
Let $H$ be a Hopf algebra. There exists a partial action of $H$ on the subalgebra $A\subseteq H_{par}$, such that
$H_{par} \simeq \underline{A\# H}$. 
\end{thm}

\begin{proof} Define, for each $h\in H$ and for each $a=\varepsilon_{h^1}\ldots \varepsilon_{h^n}\in A$ the element
\[
h\cdot a=[h_{(1)}] a[S(h_{(2)}] \in H_{par} .
\]
This is an element of $A$. Indeed,
\beqnast
h\cdot a & = & [h_{(1)}] \varepsilon_{k^1}\ldots \varepsilon_{k^n}[S(h_{(2)})] \\
& = & \varepsilon_{h_{(1)}k^1}[h_{(2)}]\varepsilon_{k^2}\ldots \varepsilon_{k^n} 
[S(h_{(3)})]\\
& = & \varepsilon_{h_{(1)}k^1}\ldots \varepsilon_{h_{(n)}k^n}[h_{(n+1)}] 
[S(h_{(n+2)})]\\
& = & \varepsilon_{h_{(1)}k^1}\ldots \varepsilon_{h_{(n)}k^n}\varepsilon_{h_{(n+1)}} \in A .
\eqnast

{\it Claim:} The linear map
\[
\begin{array}{rccc} \cdot: & H\otimes A & \rightarrow & A\\
\,  & h\otimes a & \mapsto & h\cdot a \end{array}
\]
for $h\cdot a$ as above described, defines a partial action of $H$ on $A$:

First, if $h=1_H$, then
\[
1_H \cdot a =[1_H ] a [S(1_H )]=[1_H ]a [1_H ]=1_{H_{par}} a 1_{H_{par}}=a\, , \quad \forall a\in A. 
\]

Consider $h\in H$, and two elements $a,b\in A$, then
\beqnast
h\cdot (ab) & = & [h_{(1)}] ab [S(h_{(2)})] \\
& = & [h_{(1)}][S(h_{(2)})][h_{(3)}]ab[S(h_{(4)})] \\
& = & [h_{(1)}]\te_{h_{(2)}} ab [S(h_{(3)})]\\
& = & [h_{(1)}]a\te_{h_{(2)}} b [S(h_{(3)})]\\
& = & [h_{(1)}]a[S(h_{(2)})][h_{(3)}]b[S(h_{(4)})] \\
& = & (h_{(1)}\cdot a)(h_{(2)}\cdot b) .
\eqnast
Note that the commutation in the fourth equality follows from the fact that $a\in A$ and then commutes with every element of $\tilde{A}$.

Finally, take $h,k\in H$ and $a\in A$, the composition reads
\beqnast
h\cdot (k\cdot a) & = & [h_{(1)}][k_{(1)}]a[S(k_{(2)})][S(h_{(2)})] \\
& = & [h_{(1)}][k_{(1)}][S(k_{(2)})][k_{(3)}]a[S(k_{(4)})][S(h_{(2)})]\\
& = & [h_{(1)}][k_{(1)}]\te_{k_{(2)}}a[S(k_{(3)})][S(h_{(2)})]\\
& = & [h_{(1)}][k_{(1)}] a\te_{k_{(2)}}[S(k_{(3)})][S(h_{(2)})]\\
& = & [h_{(1)}][k_{(1)}]a[S(k_{(2)})][k_{(3)}][S(k_{(4)})][S(h_{(2)})]\\
& = & [h_{(1)}][k_{(1)}]a[S(k_{(2)})][k_{(3)}][S(k_{(4)})S(h_{(2)})]\\
& = & [h_{(1)}][k_{(1)}] a\te_{k_{(2)}}[S(k_{(3)})S(h_{(2)})]\\
& = & [h_{(1)}][k_{(1)}] \te_{k_{(2)}}a[S(k_{(3)})S(h_{(2)})]\\
& = & [h_{(1)}][k_{(1)}] a[S(k_{(2)})S(h_{(2)})]\\
& = & [h_{(1)}][S(h_{(2)})][h_{(3)}][k_{(1)}] a[S(k_{(2)})S(h_{(4)})]\\
& = & [h_{(1)}][S(h_{(2)})][h_{(3)}k_{(1)}] a[S(k_{(2)})S(h_{(4)})]\\
& = & (h_{(1)}\cdot \um )(h_{(2)}k\cdot a).
\eqnast
Similarly, one can prove also that $h\cdot (k\cdot a)=(h_{(1)}k\cdot a)(h_{(2)}\cdot \um )$, therefore the partial action is symmetric.

Now, using the universal properties of $H_{par}$ and of $\underline{A\# H}$, we will have the desired isomorphism. First, we have already proved that the linear map $\pi_0 :H\rightarrow \underline{A\# H}$ given by $\pi_0 (h)=\um \# h$ is a partial representation of the Hopf algebra $H$. Then, there is an algebra morphism $\hat{\pi}:H_{par}\rightarrow \underline{A\# H}$ such that $\pi_0 =\hat{\pi}\circ [\underline{\; }]$. In order to write this morphism explicitly, note that 
$x=[h^1]\ldots [h^n]\in H_{par}$ can be written as 
\[
x=\varepsilon_{h^1_{(1)}}\varepsilon_{h^1_{(2)}h^2_{(1)}}\ldots 
\varepsilon_{h^1_{(n-1)}h^2_{(n-2)}\ldots h^{n-1}_{(1)}}[h^1_{(n)}h^2_{(n-1)}\ldots h^{n-1}_{(2)}h^n] .
\]

Therefore, we have explicitly
\[
\hat{\pi} (x)=\varepsilon_{h^1_{(1)}}\varepsilon_{h^1_{(2)}h^2_{(1)}}\ldots 
\varepsilon_{h^1_{(n-1)}h^2_{(n-2)}\ldots h^{n-1}_{(1)}}\# h^1_{(n)}h^2_{(n-1)}\ldots h^{n-1}_{(2)}h^n .
\]

On the other hand, the inclusion map $i:A\rightarrow H_{par}$ and the canonical partial representation $[\underline{\; }]:H\rightarrow H_{par}$ form a covariant pair, just because
\[
h\cdot a=[h_{(1)}]a[S(h_{(2)})] .
\]
and every $a\in A$ commutes with every $\te_h \in H_{par}$. Then, we have an algebra morphism $\Phi :\underline{A\# H}\rightarrow H_{par}$ given by
$\Phi (a\# h)=a[h]$. By construction, $\Phi$ is the inverse map of $\hat{\pi}$, therefore, these two algebras are isomorphic.
\end{proof}

\subsection{$H_{par}$ as a Hopf algebroid}

Given a Hopf algebra $H$ with invertible antipode, one can prove yet that the partial ``Hopf'' algebra is in fact a Hopf algebroid. 

\begin{defi} \cite{B} Given a $k$ algebra $A$, a left (resp. right) bialgebroid over $A$ is given by the data $(\mathcal{H}, A, s,t,\ud_l , \ue_l )$ (resp. $(\mathcal{H}, A, \tilde{s}, \tilde{t}, \ud_r , \ue_r )$) such that:
\begin{enumerate}
\item $\mathcal{H}$ is a $k$ algebra.
\item The map $s$ (resp. $\tilde{s}$) is a morphism of algebras between $A$ and $\mathcal{H}$, and the map $t$ (resp. $\tilde{t}$) is an anti-morphism of algebras between $A$ and $\mathcal{H}$. Their images commute, that is, for every $a,b\in A$ we have $s(a)t(b)=t(b)s(a)$ (resp. $\tilde{s}(a) \tilde{t}(b)=\tilde{t}(b)\tilde{s}(a)$).
By the maps $s,t$ (resp. $\tilde{s} , \tilde{t}$) the algebra $\mathcal{H}$ inherits a structure of $A$ bimodule given by $a\triangleright h \triangleleft b =s(a)t(b)h$ (resp. $a\triangleright h \triangleleft b =h\tilde{s}(b)\tilde{t}(a)$).
\item The triple $(\mathcal{H},\ud_l , \ue_l )$ (resp. $(\mathcal{H}, \ud_r , \ue_r )$) is an $A$ coring relative to the structure of $A$ bimodule defined by $s$ and $t$ (resp. $\tilde{s}$, and $\tilde{t}$).
\item The image of $\ud_l$ (resp. $\ud_r $) lies on the Takeuchi subalgebra
\[
\mathcal{H}\times_A \mathcal{H} =\{ \sum_i h_i \otimes k_i \in \mathcal{H}\otimes_A \mathcal{H} \, |\, \sum_i h_i t(a) \otimes k_i =\sum_i h_i \otimes k_i s(a) \, \; \forall a\in A \} ,
\]
resp.
\[
\mathcal{H} {_A\times} \mathcal{H} =\{ \sum_i h_i \otimes k_i \in \mathcal{H}\otimes_A \mathcal{H} \, |\, \sum_i
\tilde{s} (a) h_i  \otimes k_i =\sum_i h_i \otimes \tilde{t} (a) k_i \, \; \forall a\in A \} ,
\]
and it is an algebra morphism.
\item For every $h,k\in \mathcal{H}$, we have
\[
\ue_l (hk)=\ue_l (hs(\ue_l (k)))=\ue_l (ht(\ue_l (k))) ,
\]
resp.
\[
\ue_r (hk) =\ue_r (\tilde{s}(\ue_r (h))k)=\ue_r (\tilde{t} (\ue_r (h))k) .
\]
\end{enumerate}

Given two anti-isomorphic algebras $A$ and $\tilde{A}$ (ie, $A\cong \tilde{A}^{op}$), a left $A$ bialgebroid $(\mathcal{H}, A, s,t,\ud_l , \ue_l )$ and a right $\tilde{A}$ bialgebroid $(\mathcal{H}, \tilde{A}, \tilde{s}, \tilde{t}, \ud_r , \ue_r )$, a Hopf algebroid structure on $\mathcal{H}$ is given by an antipode, that is, an anti algebra homomorphism $\mathcal{S}:\mathcal{H}\rightarrow \mathcal{H}$ such that
\begin{enumerate}
\item[(i)] $s\circ \ue_l \circ \tilde{t} =\tilde{t}$, $t\circ \ue_l \circ \tilde{s} =\tilde{s}$, $\tilde{s}\circ \ue_r \circ t =t$ and $\tilde{t}\circ \ue_r \circ s =s$.
\item[(ii)] $(\ud_l \otimes_{\tilde A} I)\circ \ud_r =(I\otimes_A \ud_r )\circ \ud_l$ and
 $\quad (I\otimes_{\tilde A} \ud_l )\circ \ud_r =(\ud_r \otimes_A I )\circ \ud_l$. 
\item[(iii)] $\mathcal{S}(t(a)h\tilde{t}(b'))=\tilde{s}(b')\mathcal{S}(h) s(a)$,
for all $a\in A$, $b'\in \tilde{A}$ and $h\in \mathcal{H}$.
\item[(iv)] $\mu_{\mathcal{H}} \circ (\mathcal{S} \otimes I)\circ \ud_l =\tilde{s} \circ \ue_r$ and $\mu_{\mathcal{H}} \circ (I\otimes S)\circ \ud_r =s\circ \ue_l$.
\end{enumerate} 
\end{defi}

With this definition at hand let us show that the universal partial ``Hopf'' algebra $H_{par}$ is indeed a Hopf
algebroid. First, note that the antipode and its inverse in the Hopf algebra $H$, $S^{\pm 1}:H\rightarrow H$ can define
an anti-algebra maps $\mathcal{S}: H_{par}\rightarrow H_{par}$ and
$\mathcal{S}':H_{par}\rightarrow H_{par}$ given, respectively by
\beqnast
\mathcal{S}([h^1]\ldots [h^n ]) & = & [S(h^n)]\ldots [S(h^1)], \\  
\mathcal{S}'([h^1]\ldots [h^n ]) & = &[S^{-1}(h^n)]\ldots [S^{-1}(h^1)] .
\eqnast

The map $\mathcal{S}'$ (resp. $\mathcal{S}$) provides two pairs of anti-algebra isomorphisms between the subalgebras $A$ and $\tilde{A}$ of $H_{par}$ with inverse  $\mathcal{S}$ (resp. $\mathcal{S}'$), therefore we can define two kinds of source
and target maps: Relative to the algebra $A$, the maps $s,t:A\rightarrow H_{par}$ given by $s(a)=a$ and $t(a)=\mathcal{S}' (a)$,for $a\in A$, and relative to the algebra $\tilde{A}$, the maps $\tilde{s}, \tilde{t}:
\tilde{A}\rightarrow H_{par}$ given by $\tilde{s}(a')=a'$ and 
$\tilde{t}(a')=\mathcal{S}' (a' )$, for $a' \in \tilde{A}$. 

Next, we are going to provide two bialgebroid structures on $H_{par}$, the left one, relative to the base algebra $A$,
and the right one, relative to the base algebra $\tilde{A}$. For the left part, we define the $A$ bimodule structure on
$H_{par}$ by 
\[
a\triangleright x \triangleleft b =s(a)t(b) x
\]
Of course, the maps $s$ and $t$ commute among themselves because $s$ is basically the inclusion of $A$ and $t$ is the inclusion of $\tilde{A}$, and we have already proved that these two subalgebras commute.

The comultiplication and the counit are given by
\begin{eqnarray*}
\ud_l ([h^1]\ldots [h^n]) & = & [h^1_{(1)}]\ldots [h^n_{(1)}]\otimes_A [h^1_{(2)}]\ldots [h^n_{(2)}] \\
\ue_l ([h^1]\ldots [h^n]) & = & \ve_{h^1_{(1)}} \ve_{h^1_{(2)}h^2_{(1)}} \ldots \ve_{h^1_{(n)} \ldots h^n}
\end{eqnarray*}

For the right part, we define the $\tilde{A}$ bimodule structure on $H_{par}$ by 
\[
a\triangleright x \triangleleft b =x \tilde{s}(b)\tilde{t}(a) 
\]
Again, the maps $\tilde{s}$ and $\tilde{t}$ commute.

The comultiplication and the counit are given by
\begin{eqnarray*}
\ud_r ([h^1]\ldots [h^n]) & = & [h^1_{(1)}]\ldots [h^n_{(1)}]\otimes_{\tilde A} [h^1_{(2)}]\ldots [h^n_{(2)}] \\
\ue_r ([h^1]\ldots [h^n]) & = & \te_{h^1 h^2_{(1)}\ldots h^n_{(1)}} \te_{h^2_{(2)} \ldots h^n_{(2)}} \ldots
\te_{h^n_{(n)}}
\end{eqnarray*}

And the antipode, which interchanges the left and the right bialgebroid structures is given by the map $\mathcal{S}$.

Then we have the following result: 

\begin{thm} 
Let $H$ be a Hopf algebra with invertible antipode. Then, the data $(H_{par}, A, \tilde{A}, s,t,\tilde{s}, \tilde{t},
\ud_l, \ud_r, \ue_l ,\ue_r , \mathcal{S} )$ define a Hopf algebroid structure on $H_{par}$.
\end{thm}

\begin{proof} Let us first construct the comultiplication $\ud_l$. Consider the
linear map
\[
\pi \circ [\underline{\; }]\otimes [\underline{\; }]\circ \Delta :H\rightarrow H_{par}\otimes_A H_{par},
\]
where $\pi : H_{par}\otimes H_{par} \rightarrow H_{par}\otimes_A H_{par}$ is the canonical projection. 
Our aim is to show that the image of this map lies in the Takeuchi product $H_{par}\times_A H_{par}$. As this Takeuchi
product is an algebra, it suffices to verify that the map $H\to H_{par}\times_A H_{par}$ is a partial representation in
order to obtain the needed algebra map $\ud_l:H_{par}\to H_{par}\times_A H_{par}$ by the universal property of
$H_{par}$.

So first, we need to check that $\mbox{Im}(\ud_l )$ is in the Takeuchi subalgebra defined as
\[
H_{par} \times_A H_{par} =\{ \sum_i x^i \otimes_A y^i \in H_{par} \otimes_A H_{par} 
\, | \, \sum_i x^i t(a) \otimes_A y^i =\sum_i x^i \otimes_A y^i s(a)\; ,\forall a\in A \} .
\]
Indeed, take $h\in H$ and $\ve_k \in A$,
\beqnast
[h_{(1)}]\otimes_A [h_{(2)}]s(\ve_k ) & = & [h_{(1)}]\otimes_A [h_{(2)}][k_{(1)}][S(k_{(2)})] 
 =  [h_{(1)}]\otimes_A [h_{(2)}k_{(1)}][S(k_{(2)})] \\
&&\hspace{-2cm} =  [h_{(1)}]\otimes_A \ve_{h_{(2)}k_{(1)}}[h_{(3)}k_{(2)}][S(k_{(3)})] 
 =  [h_{(1)}]\otimes_A \ve_{h_{(2)}k_{(1)}}[h_{(3)}k_{(2)}S(k_{(3)})] \\
&&\hspace{-2cm} =  [h_{(1)}]\otimes_A \ve_{h_{(2)}k}[h_{(3)}] 
 =  \te_{S^{-1}(h_{(2)}k)}[h_{(1)}]\otimes_A [h_{(3)}] \\
&&\hspace{-2cm} =  [h_{(3)}k_{(2)}][S^{-1}(h_{(2)}k_{(1)})][h_{(1)}]\otimes_A [h_{(4)}] \\
&&\hspace{-2cm} =  [h_{(3)}k_{(2)}][S^{-1}(k_{(1)})S^{-1}(h_{(2)})h_{(1)}]\otimes_A [h_{(4)}] \\
&&\hspace{-2cm} =  [h_{(1)}k_{(2)}][S^{-1}(k_{(1)})]\otimes_A [h_{(2)}] 
 =  [h_{(1)}][k_{(2)}][S^{-1}(k_{(1)})]\otimes_A [h_{(2)}] \\
&&\hspace{-2cm} =  [h_{(1)}]\te_{S^{-1}(k)}\otimes_A [h_{(2)}] 
 =  [h_{(1)}]\te_{S^{-1}(k)}\otimes_A [h_{(2)}] \\
&&\hspace{-2cm} =  [h_{(1)}]t(\ve_{k} )\otimes_A [h_{(2)}] .
\eqnast

The long list of checks that $\ud_l$ satisfies the axioms of a partial representation follows the same kind of
reasoning, for example, for (PR2) we have
\beqnast
\ud_l ([h][k_{(1)}][S(k_{(2)})]) & = & [h_{(1)}][k_{(1)}][S(k_{(4)})]\ot_A [h_{(2)}][k_{(2)}][S(k_{(3)})] \\
&&\hspace{-2cm} =  [h_{(1)}][k_{(1)}][S(k_{(3)})]\ot_A [h_{(2)}]s(\ve_{k_{(2)}}) 
 =  [h_{(1)}]t(\ve_{k_{(2)}})[k_{(1)}][S(k_{(3)})]\ot_A [h_{(2)}] \\
&&\hspace{-2cm} =  [h_{(1)}]\te_{S^{-1}(k_{(2)})}[k_{(1)}][S(k_{(3)})]\ot_A [h_{(2)}] \\
&&\hspace{-2cm} =  [h_{(1)}][k_{(3)}][S^{-1}(k_{(2)})][k_{(1)}][S(k_{(4)})]\ot_A [h_{(2)}] \\
&&\hspace{-2cm} = [h_{(1)}][k_{(1)}][S(k_{(2)})]\ot_A [h_{(2)}] 
 =  [h_{(1)}k_{(1)}][S(k_{(2)})]\ot_A [h_{(2)}] \\
&&\hspace{-2cm} =  \te_{S^{-1}(h_{(2)}k_{(2)})}[h_{(1)}k_{(1)}][S(k_{(3)})]\ot_A [h_{(3)}] 
 =  [h_{(1)}k_{(1)}][S(k_{(3)})]\ot_A \ve_{h_{(2)}k_{(2)}}[h_{(3)}] \\
&&\hspace{-2cm} =  [h_{(1)}k_{(1)}][S(k_{(4)})]\ot_A [h_{(2)}k_{(2)}][S(h_{(3)}k_{(3)})][h_{(4)}] \\
&&\hspace{-2cm} =  [h_{(1)}k_{(1)}][S(k_{(4)})]\ot_A [h_{(2)}k_{(2)}][S(k_{(3)})S(h_{(3)})h_{(4)}] \\
&&\hspace{-2cm} =  [h_{(1)}k_{(1)}][S(k_{(4)})]\ot_A [h_{(2)}k_{(2)}][S(k_{(3)})] 
 =  \ud_l ([hk_{(1)}][S(k_{(4)})])
\eqnast
Therefore, the comultiplication is well defined on $H_{par}$.

Next, we have to
prove that $\ud_l$ is a morphism of $A$ bimodules. Let $h\in H$ and $\ve_k \in A$, then, we have
\beqnast
\ud_l (\ve_k \triangleright [h] ) & = & \ud_l ([k_{(1)}][S(k_{(2)})][h] ) 
 =  [k_{(1)}][S(k_{(4)})][h_{(1)}]\ot_A [k_{(2)}][S(k_{(3)})][h_{(2)}]\\
& = & [k_{(1)}][S(k_{(3)})][h_{(1)}]\ot_A \ve_{k_{(2)}} \triangleright [h_{(2)}]
 =  [k_{(1)}][S(k_{(3)})][h_{(1)}] \triangleleft \ve_{k_{(2)}} \ot_A [h_{(2)}]\\
& = & \te_{S^{-1}(k_{(2)})}[k_{(1)}][S(k_{(3)})][h_{(1)}]  \ot_A [h_{(2)}] \\
& = & [k_{(3)}][S^{-1}(k_{(2)})][k_{(1)}][S(k_{(4)})][h_{(1)}]  \ot_A [h_{(2)}]\\
& = & [k_{(3)}][S^{-1}(k_{(2)})k_{(1)}][S(k_{(4)})][h_{(1)}]  \ot_A [h_{(2)}]\\
& = & [k_{(1)}][S(k_{(2)})][h_{(1)}]  \ot_A [h_{(2)}]
 =  \ve_{k}\triangleright [h_{(1)}]  \ot_A [h_{(2)}]
 =  \ve_{k}\triangleright \ud_l ([h])
\eqnast
and by an analogous argument, one can prove that 
\[
\ud_l ([h]\triangleleft \ve_k ) = \ud_l ([h])\triangleleft \ve_k .
\]

 By construction, the comultiplication $\ud_l$ is obviously
coassociative.

The left counit,
\[
\ue_l ([h^1]\ldots [h^n]) =\ue_l (\ve_{h^1_{(1)}}\ldots \ve_{h^1_{(n-1)}\ldots h^{n-1}_{(1)}} [h^1_{(n)}\ldots h^n]) =\ve_{h^1_{(1)}}\ldots \ve_{h^1_{(n-1)}\ldots h^{n-1}_{(1)}} \ve_{h^1_{(n)}\ldots h^n} ,
\]
is, by construction, left $A$ linear. In order to prove that it is right $A$ linear, consider $[h]\in H_{par}$ and $\ve_k \in A$, then 
\beqnast
\ue_l ([h]\triangleleft \ve_k ) & = & \ue_l (\te_{S^{-1}(k)}[h]) = 
\ue_l ([k_{(2)}][S^{-1}(k_{(1)})][h])=  \ue_l ([k_{(2)}][S^{-1}(k_{(1)})h]) \\
& = & \ue_l (\ve_{k_{(2)}}[k_{(3)}S^{-1}(k_{(1)})h]) =\ve_{k_{(2)}}\ve_{k_{(3)}S^{-1}(k_{(1)})h} \\
& = & [k_{(3)}][S(k_{(4)})][k_{(5)}S^{-1}(k_{(2)})h_{(1)}]
[S(k_{(6)}S^{-1}(k_{(1)})h_{(2)})]\\
& = & [k_{(3)}][S^{-1}(k_{(2)})h_{(1)}][S(S^{-1}(k_{(1)})h_{(2)})S(k_{(4)})]\\
& = & [k_{(3)}][S^{-1}(k_{(2)})h_{(1)}][S(S^{-1}(k_{(1)})h_{(2)})][S(k_{(4)})]\\
& = & [h_{(1)}][S(h_{(2)})k_{(1)}][S(k_{(2)})]\\
& = & [h_{(1)}][S(h_{(2)})][k_{(1)}][S(k_{(2)})] =\ve_h \ve_k \\
& = & \ue_l ([h])\ve_k .
\eqnast

Because of the left $A$ linearity of $\ud_l$ and $\ue_l$, the counit axiom is easily verified, because it is needed to
check only in elements of the form $[h]\in H_{par}$.

Let us verify the property of the left counit,
\[
\ue_l (xy)=\ue_l (xs(\ue_l (y))) =\ue_l (xt(\ue_l (y))), \qquad \forall x,y\in H_{par} .
\]
Because of the left $A$ linearity, it is enough to verify for the case of $x=[h]$ and $y=[k^1 ]\ldots [k^n ]$. First, note that, for $[h]\in H_{par}$ and $a\in A$, using the isomorphism of $H_{par}$ with $\underline{A\#H }$,
\[
[h]s(a)=(\um \# h)(a\# 1_H) =((h_{(1)}\cdot a)\# h_{(2)}) =(h_{(1)}\cdot a)[h_{(2)}] ,
\]
then
\[
\ue_l ([h]s(a))=\ue_l ((h_{(1)}\cdot a)[h_{(2)}])=(h_{(1)}\cdot a)\ve_{h_{(2)}} =h\cdot a .
\]
Now, for $x=[h]$ and $y=[k^1]\ldots [k^n ]$, we have
\beqnast
\ue_l ([h]s(\ue_l (y))) & = & h\cdot \ue_l ([k^1]\ldots [k^n ]) \\
& = & [h_{(1)}] \ve_{k^1_{(1)}}\ldots \ve_{k^1_{(n)}\ldots k^n} [S(h_{(2)})]\\
& = & \ve_{h_{(1)}k^1_{(1)}}\ldots \ve_{h_{(n)}k^1_{(n)}\ldots k^n}\ve_{h_{(n+1)}}\\
& = & \ve_{h_{(1)}}\ve_{h_{(2)}k^1_{(1)}}\ldots \ve_{h_{(n+1)}k^1_{(n)}\ldots k^n}\\
& = & \ue_l (\ve_{h_{(1)}}\ve_{h_{(2)}k^1_{(1)}}\ldots [h_{(n+1)}k^1_{(n)}\ldots k^n])\\
& = & \ue_l ([h][k^1]\ldots [k^n ]) =\ue_l ([h]y) .
\eqnast
For the second equality, first note that, for $[h]\in H_{par}$ and $\ve_k \in A$,
\beqnast
\ue_l ([h]t(\ve_k )) & = & \ue_l ([h]\te_{S^{-1}(k)})=
\ue_l (\te_{S^{-1}(k)S^{-1}(h_{(2)})} [h_{(1)}])\\
& = & \ue_l (\te_{S^{-1}(h_{(2)}k)} [h_{(1)}])=
\ue_l ([h_{(1)}]\triangleleft \ve_{h_{(2)}k})\\
& = & \ue_l ([h_{(1)}]) \ve_{h_{(2)}k} .
\eqnast
Then,
\beqnast
\ue_l ([h]t(\ue_l (y))) & = & 
\ue_l ([h]t(\ve_{k^1_{(1)}}\ldots \ve_{k^1_{(n)}\ldots k^n})) \\
& = & \ue_l ([h]t(\ve_{k^1_{(n)}\ldots k^n})\ldots t(\ve_{k^1_{(1)}})) \\
& = & \ue_l ([h_{(1)}]) \ve_{h_{(2)}k^1_{(1)}}\ldots \ve_{h_{(n+1)}k^1_{(n)}\ldots k^n}\\
& = & \ue_l ([h]y)
\eqnast

Therefore the data $(H_{par}, A, s,t,\ud_l , \ue_l )$ define a structure of a left $A$ bialgebroid. With analogous
techniques, one can prove that the data $( H_{par}, \tilde{A}, \tilde{s},\tilde{t},\ud_r , \ue_r )$ is a right $A$
bialgebroid structure on $H_{par}$.

Finally, the antipode axioms give, for $x=[h^1]\ldots [h^n]$,
\beqnast
x_{(1)}\mathcal{S}(x_{(2)}) & = & [h^1_{(1)}]\ldots [h^n_{(1)}][S(h^n_{(2)})]\ldots [S(h^1_{(2)})]\\
& = & [h^1_{(1)}]\ldots \ve_{h^n}[S(h^{n-1}_{(2)})]\ldots [S(h^1_{(2)})]\\
& = & [h^1_{(1)}]\ldots [h^{n-1}_{(1)}][S(h^{n-1}_{(2)})]
\ve_{h^{n-1}_{(3)}h^n}[S(h^{n-2}_{(2)})]\ldots [S(h^1_{(2)})]\\
& = & [h^1_{(1)}]\ldots \ve_{h^{n-1}_{(1)}}[S(h^{n-2}_{(2)})]\ldots [S(h^1_{(2)})]
\ve_{h^1_{(3)}\ldots h^{n-1}_{(2)} h^n}\\
& = & \ve_{h^1_{(1)}}\ldots \ve_{h^1_{(n)}\ldots h^{n-1}_{(2)} h^n}\\
& = & s(\ue_l (x)) ,
\eqnast
similarly, we get
\[
\mathcal{S}(x_{(1)}) x_{(2)} =\tilde{s}(\ue_r (x)) .
\]
Therefore, $H_{par}$ is an $A$-Hopf algebroid.
\end{proof}

Our construction that associates to each Hopf algebra $H$, its partial ``Hopf'' algebra $H_{par}$ is in fact
functorial, as we proof in next proposition.

\begin{prop}\prlabel{Hoidfunctor}
The construction of the partial Hopf algebra defines a functor $F: \underline{\mbox{\bf{HopfAlg}}} \rightarrow
\underline{\mbox{\bf{HopfAlgbd}}}$.
\end{prop}

\begin{proof} 
On objects, the functor $F$ is described by $F(H)=H_{par}$. Now let $\varphi: H \mapsto L$ be a morphism of Hopf
algebras. By composing $\phi$  with the partial
representation $[\underline{\; } ] : L \rightarrow L_{par}$ we obtain  a partial representation of $H$ on $L_{par}$.
Therefore, we have a unique algebra morphism $\varphi_{par} : H_{par} \rightarrow L_{par}$ which is defined on
generators by $\varphi_{par} ([h]) = [\varphi(h)]$. We then define $F(\varphi)=\varphi_{par}$. It is easy to see that,
for the composition of morphisms $\varphi :H\rightarrow L$ with $\psi :L \rightarrow K$, we have 
$(\psi\circ \varphi)_{par}=\psi_{par} \circ \varphi_{par}$ and that $(\mbox{Id}_H)_{par} =\mbox{Id}_{H_{par}}$. Besides
that, $\varphi_{par}$ restricts to an algebra morphism  $\varphi_{par} : A(H_{par}) \rightarrow A(L_{par})$, where
$A(H_{par})$ is the base algebra for the Hopf algebroid structure of $H_{par}$. 
Finally, it is straightforward to prove that $\varphi_{par}$ is
a morphism of Hopf algebroids. For instance, to prove the preservation of the left comultiplication we proceed as
follows
\begin{eqnarray*}
(\varphi_{par} \otimes \varphi_{par})\ud_{H_{par}} ([h]) & = & \varphi_{par} [h_{(1)}]\otimes_{A(L_{par})} \varphi_{par}
[h_{(2)}] \\
& = & [\varphi (h_{(1)})]\ot_{A(L_{par})} [\varphi (h_{(2)})] \\
& = & [\varphi (h)_{(1)}]\ot_{A(L_{par})} [\varphi (h)_{(2)}] \\
& = & \ud_{L_{par}} ([\varphi (h)])\\
& = & \ud_{L_{par}} (\varphi_{par} ([(h)]))
\end{eqnarray*}
And analogously for the other operations of the Hopf algebroids.
\end{proof}

\subsection{The subalgebra $A$ of $H_{par}$}

We saw that $H_{par}$ is isomorphic to the partial smash product $A \# H$, where $A$ is the
subalgebra generated by the elements $\varepsilon_h = [h_{(1)}][S(h_{(2)})]$ and $H$ acts (partially) on $A$ by
$h \cdot \varepsilon_k = \varepsilon_{h_{(1)}k}\varepsilon_{h_{(2)}}$ (see \thref{Hparisosmash}). If one knows
$A$ then one might determine $H_{par}$ using the partial smash product. However, we defined $A$ as a particular
subalgebra of $H_{par}$. Our next aim is to show that it is possible to determine $A$ on its own, without calculating
$H_{par}$ first.

Consider a partial action of a Hopf algebra $H$ on an algebra $B$. Let $\pi: H \to \End(B)$ be the corresponding partial
representation and let
$\ev_1 : \mbox{End}(B) \rightarrow B$ be the evaluation map $f \mapsto f(1)$. Consider the diagram below.

\[
\xymatrix
{
& H_{par} \ar[r]^{\ue_l} \ar[d]^{\hat{\pi}} & A \ar[d]^{\hat{\pi}_1} \\
H \ar[r]_\pi \ar[ur]^{[\ul{\ }]} & \End(B) \ar[r]_{\ev_1} & B
}
\]
where $\hat{\pi}_1$ stands for the restriction of $\hat\pi$ to $A$.

The last square commutes, since
\[
\hat{\pi}_1 (\ue_l (h)) = \hat{\pi}_1 (\varepsilon_h ) = h_{(1)} \cdot (S(h_{(2)}) \cdot 1)  = h \cdot 1 = \ev_1
\hat{\pi}([h]).
\]

Note also that if we put $e=\ev_1\circ\pi : H \rightarrow B$, then we have 
\begin{eqnarray}
e(h) &= & e(h_{(1)})e(h_{(2)})             \label{A(Hpar).eqn1}\\
e(h_{(1)}) e(h_{(2)}k)&= & e(h_{(1)}k)e(h_{(2)}) \label{A(Hpar).eqn2}
\end{eqnarray}
We are going to show that these two equations define $A\subset H_{par}$.

\begin{thm}\thlabel{Auniv}
Let $H$ be a Hopf algebra a with invertible antipode. Then the algebra $A\subset H_{par}$ satisfies the following
universal property. For any algebra $B$ endowed with a linear map $e_B:H\to B$ that satisfies the equations
\eqref{A(Hpar).eqn1} and \eqref{A(Hpar).eqn2}, then there exists a unique morphism $u:A\to B$ such that $e_B=u\circ e$. 
\[
\xymatrix
{
&  A \ar@{.>}[d]^{u} \\
H \ar[r]_{e_B} \ar[ur]^{e} & B
}
\]
\end{thm}

\begin{remark}\relabel{Auniv}
It might be clear that the (unique) algebra that satisfies the universal property as in the statement of the theorem
above, is the quotient of the tensor algebra $T(H)$ by the ideal generated by elements of the form \\
\begin{enumerate}
 \item  $1_H - 1_{T(H)}$; 
\item $h - h_{(1)} \otimes h_{(2)}$; 
\item $h_{(1)} \otimes h_{(2)}k - h_{(1)}k \otimes h_{(2)}$. 
\end{enumerate}
Let us denote this algebra by  $A_1$
\end{remark}

\begin{proof}[Proof of \thref{Auniv}]
We will show that there is a canonical algebra isomorphism of $A_1$ with $A(H_{par})$.
Let $E: H \rightarrow A_1$ be the composition of the inclusion map $H \rightarrow T(H)$ with the canonical projection
of $T(H)$ on $A_1$. 
By construction, given a linear map $e: H \rightarrow B$ from $H$ into an algebra $B$ that satisfies
\eqref{A(Hpar).eqn1} and 
\eqref{A(Hpar).eqn2}, there is a unique algebra map $u: A_1 \rightarrow B $ such $e_B=u\circ E$. 
Since the map $e : H \rightarrow A$ given by 
$e(h) = \ue_l ([h]) = \varepsilon_h $ satisfies  \eqref{A(Hpar).eqn1} and 
\eqref{A(Hpar).eqn2}, there is a unique algebra morphism $\varphi : A_1 \rightarrow A$ such that $\varphi\circ E =
e$.

On the other hand,  we can define a partial action of $H$ on $A_1$ as follows.
Consider the $k$-linear map $\Phi : H \otimes T(H) \rightarrow T(H)$ defined on homogeneous elements by 
\[
\Phi (h \otimes (k^1 \otimes \cdots \otimes k^n) = h_{(1)}k^1 \otimes \cdots \otimes h_{(n)} k^n \otimes h_{(n+1)}
\]
If $J$ is the ideal of \reref{Auniv} above then $\Phi(H \otimes J) \subset J$. For instance,
if we consider the generator 
$k -  k_{(1)} \otimes k_{(2)}$, we have
\[
\Phi (h \otimes (k - k_{(1)} \otimes k_{(2)}) ) = (h_{(1)}k   - h_{(1)(1)}k_{(1)} \otimes h_{(1)(2)}k_{(2)} ) \otimes h_{(2)} 
\]
which lies in $J$. The same happens for $1_H - 1_{T(H)}$ and for the family of generators. From the definition of
$\Phi$ it follows that the same occurs for every element of $J$. Therefore, by passing to the quotient we have a
well-defined $k$-linear map 
\begin{eqnarray*}
\overline{\Phi}:
H \otimes A_1 & \rightarrow & A_1 \\
h \otimes E(k^1) \cdots E(k^n) & \mapsto &  E(h_{(1)}k^1) \cdots E(h_{(n)}k^n) E(h_{(n+1)})
\end{eqnarray*}
Now it can be checked that this map defines a partial action of $H$ on $A_1$, and therefore we may consider also the
associated partial representation $\pi: H \rightarrow \mbox{End}(B) $, and the linear map $e = \ev_1 \pi: H \rightarrow
A_1$ defined by $h \mapsto h \cdot 1 = E(h)$.
By the universal property of $H_{par}$ (and the discussion in the beginning of this section) there is a unique algebra
morphism 
$\hat{\pi}_1: A(H_{par}) \rightarrow A_1$ such that  $\hat{\pi}(\varepsilon_h) = \hat{\pi}_1 \ue_l (h) = e \hat{\pi}
(h) = E(h)$. Now it is easy to see that $\hat{\pi}_1$ is the inverse to $\varphi: A_1 \rightarrow A(H_{par})$ obtained
above as $E(h) \mapsto \varepsilon_h$.
\end{proof}

We can use the above result to determine $H_{par}$ for some Hopf algebras first calculating the universal algebra
$A(H_{par})$ and then constructing the partial smash product.

\begin{exmp} Let $H$ be  the Sweedler Hopf algebra $H_4$.  We will see that  $A(H_{par})$ is a commutative algebra isomorphic to $k[X,Z]/I$ where $I$ is the ideal generated by $2X^2-X$ and $2XZ-Z $. In order to obtain this result we will work with the basis  $e_1 = (1+g)/2$, $e_2 = (1-g)/2$, $h_1 = xe_1$, 
$h_2 = xe_2$. This choice makes easier some computations because $e_1$ and $e_2$ are idempotent and $h_1,h_2$ generate the radical of $H_4$ (and are nilpotent of order 2). The multiplication table of $H_4$ in this new basis elements reads 
\[
\begin{array}{|c|c|c|c|c|}
\hline
 & e_1 & e_2 & h_1 & h_2 \\
\hline 
e_1 & e_1 & 0 & 0 & h_2 \\
\hline 
e_2 & 0 & e_2 & h_1 & 0  \\
\hline 
h_1 & h_1 & 0 & 0 & 0 \\
\hline 
h_2 & 0 & h_2 & 0 & 0 \\
\hline 
\end{array} 
\]
The expressions for the coproducts of this new basis, are 
\begin{eqnarray}
\Delta (e_1) &=& e_1 \otimes e-1 + e_2 \otimes e_2 , \nonumber\\
\Delta (e_2) &=& e_1 \otimes e_2 + e_2 \otimes e_1 , \nonumber\\
\Delta (h_1) &=& e_1 \otimes h_1 - e_2 \otimes h_2 + h_1 \otimes e_1 
+h_2 \otimes e_2 , \nonumber\\
\Delta (h_2) &=& e_1 \otimes h_2 - e_2 \otimes h_1 + h_1 \otimes e_2 
+h_2 \otimes e_1 .\nonumber
\end{eqnarray}
The counit calculated in the elements of this new basis takes the values
$\epsilon(e_1 ) =1$ and $\epsilon (e_2 )=\epsilon (h_1 )=
\epsilon (h_2 )=0$. Finally, the antipode  of these elements are given by
\[
S(e_1 )=e_1, \quad S(e_2 )=e_2, \quad S(h_1 )=-h_2,   \quad  
S(h_2)=h_1 .
\]

In what follows, let us call $e = \varepsilon_{e_1}$, $x = \varepsilon_{e_2}$, $y = \varepsilon_{h_1}$, $z =
\varepsilon_{h_2}$. 

We recall that the defining equations for $A(H_{par})$ are 
\begin{eqnarray}
\varepsilon{h} & = & \varepsilon_{h_{(1)}}\varepsilon_{h_{(2)}} \label{sweedler.eqn1}\\
\varepsilon_{h_{(1)}k}\varepsilon_{h_{(2)}} & = &  \varepsilon_{h_{(1)}}\varepsilon_{h_{(2)}k} \label{sweedler.eqn2}
\end{eqnarray}
for $h,k \in H$.

Writing the equations \eqref{sweedler.eqn2} for these generators one concludes that $A(H_{par})$ is commutative and
that $xy = xz = 0$.

Doing the same for equations \eqref{sweedler.eqn1} and using that $1 = e + x$ (since $1 = e_1 + e_2$ in $H_4$) one
obtains 
\[
(1 - x)  =  (1 - x )^2  + x^2
\]
\[
x  =  2(1-x)x, \ \ y  =  2(1-x)y,  \ \ z  =  2(1-x)z 
\]
and then from $xy=0$  we get $y= 2y$, and hence $y=0$. 
The second equation above also yields $2x^2=x$.  Therefore, we obtain 
\begin{equation}
y = 0, \ \ 2x^2=x, \ \ 2xz=z
\end{equation}
i.e., 
\begin{equation}
\varepsilon_{h_1} = 0, \ \ 2\varepsilon_{e_2}^2=\varepsilon_{e_2}, \ \ 2\varepsilon_{e_2}\varepsilon_{h_2}=\varepsilon_{h_2} 
\end{equation}

Since these latter equations were obtained using only \eqref{sweedler.eqn1} and \eqref{sweedler.eqn2}, if  $e: H_4
\rightarrow B $ is any linear map 
satisfying \eqref{A(Hpar).eqn1} and \eqref{A(Hpar).eqn2} from last section 
then it follows that  
\begin{equation}
e(h_1) = 0, \ \ 2e(e_2)^2=e(e_2), \ \ 2e(e_2)e(h_2)=e(h_2). \label{defining.eqn}
\end{equation}

Let $E: H_4 \rightarrow k[X,Z]/I$ be the map taking $e_1$ to $1-X$, $e_2$ to $X$, $h_1$ to $0$ and $h_2$ to $Z$. It can
be checked that this map satisfies \eqref{A(Hpar).eqn1} and \eqref{A(Hpar).eqn2}. Consider now another linear map $e :
H_4 \rightarrow B$  that satisfies these same equations. Then  $e$ also satisfies  \eqref{defining.eqn} and there is a
unique algebra morphism $f: k[X,Z]/I \rightarrow B$ such that the triangle below commutes.

\[
\xymatrix{
              &  k[X,Z]/I \ar[d]^f \\
H_4 \ar[r]^e \ar[ur]^E & B \\
}
\]
In particular, the linear map $h \in H_4 \mapsto \varepsilon_h \in A(H_{par})$ yields an algebra isomorphism with
$k[X,Z]/I$, given on generators by $X \mapsto \varepsilon_{e_2}$ and $Z \mapsto \varepsilon_{h_2}$. 

This isomorphism allows us to describe more explicitly, in this case, the action of $H_{par}$ on $A(H_{par})$, which is
given by 
$h \cdot \varepsilon_k = \varepsilon_{h_{(1)}k} \varepsilon_{h_{(2)}}$. Following this formula but using the isomorphism
above we have the partial action of $H_4$ on $k[X,Z]/I$ given by 
\begin{align*}
e_1 \cdot  1 & = (1-X), &  e_2 \cdot 1 & = X,  &  h_1 \cdot 1 &  = 0 , &  h_2 \cdot 1 &  = Z/2 \\
e_1 \cdot  X &= X/2, & e_2 \cdot X &  = X/2 , &  h_1 \cdot X &  = 0, &   h_2 \cdot X &  = Z/2 \\
e_1 \cdot  Z^n &= Z^n/2, &  e_2 \cdot Z^n & = Z^n/2, &  h_1 \cdot Z^n & = 0 ,  & h_2 \cdot Z^n & = Z^{n+1}
\end{align*}

Note that, for instance, that an algebra morphism $k[X,Z]/I \rightarrow k$ must either send both $X$ and $Z$ to zero, or
send $X$ to $1/2$ and $Z$ to any given $\alpha \in k$. In this manner we reobtain all the partial actions of $H_4$ on
the field $k$  \cite{AAB}.  

Having described $A(H_{par})$ we can now describe the whole algebra $H_{par}$, since $H_{par}$ is isomorphic to the
partial smash product $A(H_{par}) \# H$.

$A(H_{par}) \# H$ is generated, as a vector space, by elements of the form $1 \# k$, $x \# k$ and $z^n \# k$ for $k \in
\{e_1,e_2,h_1,h_2\}$ and $n \geq 1$.  From these generators we can extract a basis of  $A(H_{par}) \# H$. 

For instance, 
\[
x \# e_1 = x (e_1 \cdot 1) \# e_1 + x (e_2 \cdot 1) \# e_2  =  x^2  \# e_1 + x(1-x) \# e_2 = 
\dfrac{x}{2} \# e_1 + \dfrac{x}{2} \# e_2 = \dfrac{x}{2} \# 1
\]
and it follows that 
\[
x \# e_1 = x \# e_2 = \dfrac{x}{2} \# 1.\]
 Analogous computations yield 
\begin{align*}
z^n  \# e_1   =& z^n \#e_2  = \dfrac{z^n}{2} \#1\\
x \# h_2   =& z \# 1 - x \# h_1 \\
z^{n}  \# h_2  = &z^{n+1} \#1 - z^n\#h_1 
\end{align*}

It follows that  $\underline{A({H_4}_{par}) \# H_4}$  is generated by $1 \#e_1$, $1\#e_2$, $1\#h_1$, $1\# h_2$, $x \#
1$,
$z^n \# 1$, $x \#h_1$, $z^n \# h_1$. 
 It can be verified that this is indeed a basis by expressing  these elements in terms of the tensor product $A \otimes
H$, taking a linear dependence of them and applying appropriate linear maps  of the form $I_{A(H_{par}) } \otimes
\varphi$ to it (with $\varphi \in (H_4)^*$).
Therefore, the algebra ${H_4}_{par} \cong \underline{A({H_4}_{par})\# H_4 }$ is completely determined. 

Remark that this example gives a negative answer to the question raised in \reref{Hparfinite}, whether the partial
``Hopf'' algebra of a finite dimensional Hopf algebra is always finite dimensional. Also, as seperable-Frobenius
algebras are finite dimensional, it also shows that the base algebra of $H_{par}$ as a Hopf algebroid is not
necessarily a separable-Frobenius algebra (as is the case for the Hopf algebroid associated to a weak Hopf algebra).
\end{exmp}

\section{The category of partial modules}

Partial representations allow us to provide a correct categorical framework for partial actions. 
\begin{defi}
Let $H$ be a Hopf algebra. 
A (left) {\em partial module} over $H$ is a pair $(M,\pi)$, where $M$ is a $k$-vector space and $\pi:H\to \End_k(M)$ is a (left) partial representation of $H$. \\
If $(M,\pi)$ and $(M',\pi')$ are partial $H$-modules, then a morphism of partial $H$-modules is a $k$-linear map $f:M\to M'$ satisfying $f\circ \pi(h)= \pi'(h)\circ f$ for all $h\in H$.\\
The category with as objects partial $H$-modules and the morphisms defined above between them is denoted as ${_H\Mm}^{par}$
\end{defi}

\begin{remark}
Using the classical Hom-Tensor relations, a $k$-vector space $M$ is a partial $H$-module if and only if there exists a $k$-linear map $\bul:H\ot M\to M$ satisfying the following axioms for all $m\in M$ and all $h, k\in H$
\begin{enumerate}[{[PM1]}]
\item $1_H \bul m= m$;
\item $h\bul (k_{(1)}\bul (S(k_{(2)}\bul m))=(hk_{(1)})\bul (S(k_{(2)})\bul m)$;
\item $h_{(1)}\bul (S(h_{(2)})\bul (k\bul m))=h_{(1)}\bul (S(h_{(2)})k\bul m)$;
\item $h\bul (S(k_{(1)}) \bul (k_{(2)} \bul m)) =hS(k_{(1)}) \bul (k_{(2)} \bul m)$;
\item $S(h_{(1)})\bul (h_{(2)} \bul (k\bul m)) = S(h_{(1)})\bul (h_{(2)} k\bul m)$.
\end{enumerate}
\end{remark}

From our previous considerations we now immediately obtain the following.

\begin{cor}
Let $H$ be a Hopf algebra, then there is an isomorphism of categories
$${_H\Mm^{par}}\cong {_{H_{par}}\Mm}.$$ 
and these categories are furthermore equivalent to the category of left $\underline{A\# H}$-modules, 
${_{\underline{A\# H}}\Mm}$. 
Moreover, as $H_{par}$ has a structure of a Hopf algebroid over $A$, ${_H\Mm^{par}}$ is a closed monoidal category
admitting a strict monoidal functor that preserves internal homs
$$U :{_H\Mm^{par}}\to {_A\Mm_A}.$$
\end{cor}

\begin{proof}
The first isomorphism follows directly from \thref{univHpar}. The second statement follows from the isomorphism $H_{par}\cong \underline{A\# H}$, proven in \thref{Hparisosmash}. For the last statement, remember that $H_{par}$ is an $A$-Hopf algebroid. Hence the category of $H_{par}$-modules is monoidal with a strict monoidal forgetful functor to $A$-bimodules. By the above equivalences of categories, the statement follows.
\end{proof}

Let us make the obtained structure on the category of partial modules as in the previous corollary a bit more explicit.
Firstly, to identify the category of partial modules over $H$ as the category of modules over $\underline{A\# H}$, we put
\[
h\bullet m=[h]\triangleright m = (\um \# h)\triangleright m.
\]
At first sight, it seems like the category of partial $H$ modules would not coincide with the category of $\underline{A\# H}$ modules, but no extra structure in fact is needed, because, for any elements $(a\# h) \in \underline{A\# H}$, where $a=\varepsilon_{k^1}\ldots \varepsilon_{k^n}$, and $m\in M$ we have 
\beqnast
(a\# h)\triangleright m & = &  (\varepsilon_{k^1}\ldots \varepsilon_{k^n} \# h)\triangleright m\\
& = & (\varepsilon_{k^1}\# 1_H )\ldots (\varepsilon_{k^n}\# 1_H )(\um \# h)\triangleright m \\
& = & (\um \# k^1_{(1)})(\um \# S(k^1_{(2)}))\ldots  (\um \# k^n_{(1)})(\um \# S(k^n_{(2)}))(\um \# h) \triangleright m \\
& = & k^1_{(1)}\bullet (S(k^1_{(2)}) \bullet (\cdots ( k^n_{(1)}\bullet (S(k^n_{(2)}) \bullet (h\bullet m))) \cdots ) ) .
\eqnast
We can also conclude that every left partial $H$ module is automatically an $A$ bimodule, by the obvious left structure, 
\[
am = s(a)\triangleright m=(a\# 1_H )\triangleright m ,
\]
and a more involved right structure:
\[
ma= t(a)\triangleright m.
\]
More explicitly, if $a=\ve_{h^1}\ldots \ve_{h^n}$, then $t(a)=\te_{S^{-1}(h^n)} \ldots \te_{S^{-1}(h^1)}$, or in terms of smash product elements
\[
t(a)=(\um \# h^n_{(2)})(\um \# S^{-1}(h^n_{(1)}))\ldots (\um \# h^1_{(2)})(\um \# S^{-1}(h^1_{(1)})) .
\]

Concerning the monoidal structure on the category of partial $H$-modules, recall that the unit object is the base algebra $A\subset H_{par}$ with partial action given by the counit of $\underline{A\# H}$
\beqnast
h\bullet a & = & (\um \# h)\triangleright a \\
& = & \underline{\epsilon} ((\um \# h)(a\# 1_H ))\\
& = & \underline{\epsilon} ((h_{(1)}\cdot a)\#h_{(2)})\\
& = & (h_{(1)}\cdot a)(h_{(2)}\cdot \um )\\
& = & (h\cdot a)
\eqnast

Given two objects, $M,N\in {}_H \mathcal{M}^{par}$, the action of an element $h\in H$ on an element $m\otimes_A n \in M\otimes_A N$ is defined by the comultiplication of $\underline{A\# H}$,
\[
h\bullet (m\otimes_A n) =  (h_{(1)} \bullet m) \otimes_A (h_{(2)} \bullet n)
\]

\begin{lemma} \label{algebraobjects1} Let $B$ be an algebra object in the monoidal category ${}_{H}\mathcal{M}^{par}$, then $H$ acts partially on $B$ with a symmetric action.
\end{lemma}

\begin{proof} First, let us verify that $1_H \bullet x =x$ for any $x\in B$.
\[
1_H \bullet x =(\um \# 1_H )\triangleright x =1_{\underline{A\# H}} \triangleright x =x .
\]
Then, the multiplicativity of the partial action comes naturally  that $B$ is an algebra object in ${}_{\underline{A\# H}} \mathcal{M}$, therefore, the multiplication in $B$ is a module map
\[
h\bullet (xy)=(\um \# h)\triangleright (xy)=((\um \# h_{(1)})\triangleright x)
((\um \# h_{(2)})\triangleright y)=(h_{(1)}\bullet x)(h_{(2)}\bullet y).
\]
Finally, by the property of the unit map 
$\eta \in {}_{\underline{A\# H}}\mbox{Hom} (A,B)$, we have
\beqnast
(h\bullet 1_B)x  & = & (h\bullet \eta (\um ))x = \eta (h\cdot \um )x \\
& = & \eta ( \ve_h )x =\ve_h \eta (\um )x \\
& = & \ve_h 1_Bx \\
& = & \ve_h x ,
\eqnast
where the last equality follows from the fact that the multiplication in $B$ is a morphism of $A$ bimodules.

Therefore, for $h,k\in H$ and $x\in B$, we have
\beqnast
h\bullet (k\bullet x) & = & (\um \# h)\triangleright ((\um \# k)\triangleright x) \\
& = & ((\um \# h)(\um \# k)) \triangleright x \\
& = & ((h_{(1)}\cdot \um )\# h_{(2)}k)\triangleright x \\
& = & (((h_{(1)}\cdot \um )\# 1_H)(\um \# h_{(2)}k))\triangleright x \\
& = & ((h_{(1)}\cdot \um )\# 1_H)\triangleright ((\um \# h_{(2)}k)\triangleright x) \\
& = & \ve_{h_{(1)}} (h_{(2)}k\bullet x)\\
& = & (h_{(1)}\bullet 1_B )(h_{(2)}k \bullet x)
\eqnast

On the other hand
\beqnast
(h_{(1)}k\bullet x )(h_{(2)} \bullet \um ) & = & 
(h_{(1)}k\bullet x ) (\ve_{h_{(2)}}1_B ) \\
& = & ((h_{(1)}k\bullet x )\ve_{h_{(2)}})1_B \\
& = & t(\ve_{h_{(2)}}) (h_{(1)}k\bullet x )\\
& = & (\um \# h_{(3)})(\um \# S^{-1} (h_{(2)})) \triangleright 
( (\um \# h_{(1)}k)\triangleright x )\\
& = &  (\um \# h_{(3)})(\um \# S^{-1} (h_{(2)}))(\um \# h_{(1)}k)\triangleright x \\
& = & (\um \# h_{(4)})((S^{-1} (h_{(3)})\cdot \um )\# S^{-1} (h_{(2)})h_{(1)}k)\triangleright x\\
& = & (\um \# h_{(2)})((S^{-1} (h_{(1)})\cdot \um )\# k)\triangleright x\\
& = & ( (h_{(2)}\cdot (S^{-1} (h_{(1)})\cdot \um ))\# h_{(3)}k)\triangleright x\\
& = & ( (h_{(2)}S^{-1} (h_{(1)})\cdot \um )(h_{(3)}\cdot \um )\# h_{(4)}k)\triangleright x\\
& = & ( (h_{(1)}\cdot \um )\# h_{(2)}k)\triangleright x\\
& = & (\um \# h)(\um \# k)\triangleright x\\
& = & h\bullet (k\bullet x))
\eqnast
for any $h,k\in H$ and for any $x\in B$. Therefore, the action of $H$ on $B$ is symmetric. This completes the proof.
\end{proof}

It is also easy to see that an algebra object $B\in {}_H \mathcal{M}^{par}$ has also a 
$k$ algebra structure, with multiplication $\mu_0 :B\otimes_k B \rightarrow B$ and unit $\eta_0 :k\rightarrow B$ given, respectively by $\mu_0= \mu_B \circ \pi$ and $\eta_0 =\eta_B \circ \eta_A$, where $\pi :B\otimes_k B\rightarrow B\otimes_A B$ is the canonical projection and $\eta_A$ is the unit map of the base algebra $A$. 

\begin{lemma} \label{algebraobjects2} Let $B$ be a $k$-algebra on which $H$ acts partially with a symmetric action. Then $B$ is an algebra object in the monoidal category 
$ ( {}_H \mathcal{M}^{par} , \otimes_A , A)$
\end{lemma}

\begin{proof} First, we need to show that $B$ is an $A$ algebra, that is, it is an algebra with respect to the tensor product over $A$. The structure of $A$ bimodule on $B$ is defined naturally by the partial action, in fact, given $h\in H$ and $x\in B$ we have
\[
\varepsilon_h x = h_{(1)}\bullet (S(h_{(2)}) \bullet x) =(h\bullet 1_B )x
\]
and
\[
x\ve_h = h_{(2)} \bullet (S^{-1} (h_{(1)})\bullet x =x(h\bullet 1_B ),
\]
in particular $\ve_h 1_B = 1_B \ve_h$ for all $h\in H$.

The multiplication is a morphism of $A$ bimodules:
\beqnast
\ve_h (xy) & = &  h_{(1)}\bullet (S(h_{(2)}) \bullet (xy)) \\
& = & (h_{(1)}\bullet 1_B )(h_{(2)}S(h_{(3)}) \bullet (xy))\\
& = & (h\bullet 1_B )(xy)\\
& = & ((h\bullet 1_B )x)y\\
& =& (\ve_h x)y ,
\eqnast
and 
\beqnast
(xy)\ve_h & = &  h_{(2)}\bullet (S^{-1}(h_{(1)}) \bullet (xy)) \\
& = & (h_{(2)}S^{-1}(h_{(1)}) \bullet (xy))(h_{(3)}\bullet 1_B )\\
& = & (xy)(h\bullet 1_B )\\
& = & x(y(h\bullet 1_B ))\\
& = & x(y\ve_h ),
\eqnast

We must to verify that the multiplication map in $B$ is balanced with relation to $A$, that is, given $x,y \in B$ and $h \in H$, then $(x\blacktriangleleft \varepsilon_h )y = 
x(\varepsilon_H \blacktriangleright y)$. Indeed
\beqnast
(x\varepsilon_h )y & = & (h_{(2)} \bullet (S^{-1}(h_{(1)})\bullet x))y \\
& = &  (h_{(2)} S^{-1}(h_{(1)})\bullet x)(h_{(3)}\bullet 1_B )y \\
& = &  x(\ve_h 1_B )y \\
& = & x(\varepsilon_h y) .
\eqnast

Now, define the unit map 
\[
\begin{array}{rccc}
\tilde{\eta} : & A & \rightarrow & B \\
\,             & a & \mapsto     & a 1_B
\end{array}
\]
which clearly is an $A$ bimodule map, and also satisfies the unit axiom, proving that $B$ is an $A$ algebra.

Next, one needs to show that the multiplication and the unit maps are morphisms of partial $H$ modules. The multiplicativity of the action comes from the hypothesis the $H$ acts partially on $B$. Finally, for the unit map, as it is already an $A$ bimodule map, one needs only to prove that $\widehat{\eta} (h\cdot \um ) =h\bullet 1_B$, indeed
\beqnast
\tilde{\eta}(h\cdot \um ) & = & \tilde{\eta} (\ve_h ) =  \varepsilon_h  1_B \\
& = & (h\bullet 1_B )1_B = h\bullet 1_B.
\eqnast
Therefore, $B$ is an algebra object in ${}_H \mathcal{M}^{par}$ with the monoidal structure given by the tensor product over $A$.
\end{proof}

The results proved in the Lemmas \ref{algebraobjects1} and \ref{algebraobjects2} lead immediately to the following.

\begin{teo}
There is an isomorphism  between the categories of symmetric partial $H$-module algebras and the algebra objects in the
category of partial $H$-modules, i.e. $\ParAct_H\cong \Alg({_{H_{par}}}\Mm)$.
\end{teo}

This theorem has the following immediate consequence, that answers the question of \reref{freefunctor}

\begin{cor}
The functor $\Pi$ from \thref{Pi} coincides exactly with the forgetful functor $\Alg({_{H_{par}}}\Mm)\to
{}_{H_{par}}\Mm$. As a consequence, this functor has a left adjoint, that takes the free algebra on an $H_{par}$-module.
\end{cor}

\begin{exmp} Let $G$ be a group and $kG$ be its group algebra. Because $kG$ is a cocommutative Hopf algebra, we can describe its category of partial modules using the above described construction. In this case, the subalgebra $A\subseteq k_{par} G$ is commutative, then $A\cong \tilde{A}$, $\te_h =\ve_h^{-1}$ and the morphism $\mathcal{S}'$ is the identity map on $A$. The partial smash product $\underline{A\# kG}$ coincides with the partial crossed product $A\rtimes_{\alpha} G$, presented in the first section. The partial $kG$ module structure is given by
\[
\delta_g \bullet m =(\varepsilon_g \delta_g )\triangleright m .
\]
It is easy to see that, given a left partial $kG$ module $M$, on can define a partial group action of $G$ on the $k$ module $M$ by partially defined $k$ linear homomorphisms. This partial action is given by a family $\{ X_g \}_{g\in G}$ of $k$ submodules of $M$ and a family $\{ \theta_g :X_{g^{-1}}\rightarrow X_g \}_{g\in G}$ of $k$ linear isomorphisms satisfying
\begin{enumerate}
\item[(a)] $X_e =M$ and $\theta_e =\mbox{Id}_M$.
\item[(b)] $\theta_g (X_{g^{-1}}\cap X_h )=X_g \cap X_{gh}$.
\item[(c)] For any $m\in \theta_{h^{-1}} (X_{g^{-1}}\cap X_h )$, we have $\theta_g (\theta_h (x))=\theta_{gh} (x)$.
\end{enumerate}
The submodules $X_g$ are defined from the projections $P_g :M\rightarrow M$ given by
\[
P_g (m) =\delta_g \bullet (\delta_{g^{-1}} \bullet m)=(\varepsilon_g\delta_e )\triangleright m .
\]
It is easy to see that the linear operators $P_g$ are projections and that for any $g,h\in G$, $P_g$ and $P_h$ commute.

The isomorphisms $\theta_g$ are simply the action of $\delta_g$ restricted to $X_{g^{-1}}$, that is
\[
\theta_g (P_{g^{-1}} (m))=\delta_g \bullet (P_{g^{-1}}(m)) .
\]
It is straightforward to verify that the image of $\theta_g$ is in $X_g$. The data 
$( \{ X_g \}_{g\in G} , \{ \theta_g \}_{g\in G} )$ indeed define a partial action of the group $G$ on  the $k$ module $M$ by partially defined $k$ isomorphisms.

On the other hand, given a partial action $( \{ X_g \}_{g\in G} , \{ \theta_g \}_{g\in G} )$ on a $k$-vector space $M$, define $P_g :M\rightarrow M$ as the linear projection over the subspace $X_g$. For each $g\in G$ and each $m\in M$ define
\[
\delta_g \bullet m =\theta_g (P_{g^{-1}} (m)) .
\]
It is easy to see that the map
\[
\begin{array}{rccc}
\pi : & kG                         & \rightarrow & \mbox{End}_k (M)\\
\,    & \sum_{g\in G} a_g \delta_g & \mapsto     & \sum_{g\in G} a_g \delta_g \bullet \underline{\quad}
\end{array}
\]
is a partial representation of $kG$, therefore, $M$ has a structure of a partial $kG$ module. 
\end{exmp}

In the previous example, the base field $k$ can have more than one structure of a partial $kG$ module.

\begin{prop} The structures of partial $kG$ modules on the base field $k$ are in one to one correspondence with the set of the pairs $(H,\lambda)$, in which $H\subseteq G$ is a subgroup and $\lambda :H \rightarrow k^*$ is a one dimensional representation of the subgroup $H$.
\end{prop}

\begin{proof} First, suppose that $k$ is endowed with a $kG$ partial module structure. Then, there is a partial action of the group $G$ on $k$, given by subspaces $k_g \subseteq k$ and linear isomorphisms $\theta_g : k_{g^{-1}}\rightarrow k_g$. The subspaces $k_g$ for each $g\in G$ are the images of the projections $P_g \in \mbox{End}_k (k)\cong k$, therefore, there are two possibilities, either $P_g$ is the multiplication by $1$, and $k_g \cong k$, or $P_g$ is the multiplication by $0$, and $k_g =0$. Define
\[
H=\{ g \in G \, | \, k_g =k \} .
\]
As we have a partial action, then $k_e =k$, then $e\in H$. Also, if $g\in H$, as $\theta_{g^{-1}}:k_g \rightarrow k_{g^{-1}}$ is a $k$ linear isomorphism, then $k_{g^{-1}} \cong k$, which leads to the conclusion that $g^{-1}\in H$. Moreover, if $g,h\in H$ then $k_g$, $k_{g^{-1}}$, $k_h$ and $k_{h^{-1}}$ are all isomorphic to $k$, then
\[
k\cong \theta_g (k_{g^{-1}} \cap k_h )=k_g \cap k_{gh} \cong k\cap k_{gh} =k_{gh} ,
\]
which means that $gh\in H$. Therefore, $H$ is a subgroup of $G$. Finally, for $g,h\in H$, then we have that $\theta_g \circ \theta_h =\theta_{gh}$, by the partial action axiom. As $\mbox{End}_k (k) \cong k$ for each $g\in G$ there is a number $\lambda_g \in k$ such that the linear isomorphism $\theta_g$ is given by $\theta_g (x)=\lambda_g x$, and $\lambda_g \lambda_h =\lambda_{gh}$. By the fact that $\theta_g$ is an isomorphism, we have that $\lambda_g \in k^*$. Therefore, the map 
\[
\begin{array}{rccc} 
\lambda : & G & \rightarrow & k^* \\
\,       & g & \mapsto     & \lambda_g
\end{array}
\]
is a one dimensional linear representation of the group $H$.

On the other hand, given a subgroup $H\subseteq G$ and a one dimensional linear representation $\lambda :H\rightarrow k^*$, one can define a partial action of $G$ on $k$ by the following data: The subspaces $k_g$ are equal to $k$ for $g\in H$ and $k_g =0$ for $g\notin H$, and the isomorphisms $\theta_g :k_{g^{-1}}\rightarrow k_{g}$ given by, $\theta_g (x) =\lambda_g x$ for $g\in H$ and $\theta_g =0$ for $g\notin H$. This, indeed, constitutes a partial action of the group $G$ on the base field $k$. 
\end{proof}

\begin{exmp} For the case where $H=\mathcal{U}(\mathfrak{g})$, we have proved earlier that $H_{par} \cong H$, therefore ${}_{\mathcal{U}(\mathfrak{g})}\mathcal{M}^{par} = {}_{\mathcal{U}(\mathfrak{g})_{par}}\mathcal{M} \cong {}_{\mathcal{U}(\mathfrak{g})}\mathcal{M}$.
\end{exmp}

\section{Partial $G$-Gradings}

Recall the well-known fact that the category of $k$-vector spaces graded over a group $G$ is isomorphic to the category
of comodules over the associated group algebra $kG$. Similarly, a $G$-graded $k$-algebra is nothing else than a
$kG$-comodule algebra, i.e. an algebra object in the monoidal category of $kG$-comodules. A first example of a partially
graded algebra over a finite group $G$ was introduced in \cite{AB2}. In this case, the category of $kG$-comodules is
equivalent with the category of $kG^*$-modules, where $kG^*=k^G$ is the dual group algebra. Hence a partially $G$-graded
$k$-algebra can be introduced as an algebra in the category of partial $kG^*$-modules, or in other words as an algebra
that admits a partial action of $kG^*$. We give the explicit definition below, and will treat the case of partial
gradings by infinite groups in a later paper.

\subsection{Partially graded vector spaces and partially graded algebras}

\begin{defi} \delabel{def.partial.grading}
Let $G$ be a finite group, $k^G$ the dual group algebra. We denote by $p_g\in k^G$ the morphism such that $p_g(h)=\delta_{g,h}$, then $\{p_g, g\in G\}$ forms a canonical base for $k^G$ as a vector space. We say that a $k$-algebra $A$ admits a {\em partial $G$-grading} if the following axioms are fulfilled.
\begin{enumerate}
\item $\sum_{g \in G}  p_g \cdot a  = a$ for all $a \in A$ 
\item $p_g \cdot ab = \sum_{l \in G} (p_{gl^{-1}} \cdot a) (p_l \cdot b)$
\item $p_g \cdot ( p_t \cdot a)  =  (p_{gt^{-1}} \cdot 1_A) (p_t \cdot a) =(p_t \cdot a) (p_{t^{-1}g} \cdot 1_A)$
\end{enumerate}
\end{defi}

As said  before, the category of partial $k^G$-modules is basically the category of partially $G$ graded vector spaces.
Therefore, a partially $G$-graded vector space, is nothing else than a module over $(k^G)_{par}\cong
\underline{\hat{A}\# k^G}$, where $\hat{A}$ is the subalgebra of $(k^G)_{par}$ generated by elements of the form
\[
\varepsilon_{p_g} =\sum_{h\in G} [ p_h ][p_{g^{-1}h}] .
\]

Instead of dealing with partially $G$-graded spaces in this level of generality, it is more useful to present some more
concrete examples. More precisely, we will present examples of partial $G$-grading on the $k$-vector space $k^n$ (i.e.
structures of $k^G$-module on $k^n$) that are induced by a partial grading on the $k$-algebra
$\End_k(k^n)=M_n(k)$ (i.e. structures of partial $k^G$-actions on $M_n(k)$, as these imply partial representations on
$M_n(k)$). 

In what follows, $k$ is a field which has  characteristic coprime to $|G|$. First we consider the partial $G$-gradings
on the base field $k$, which are characterized by the following proposition.

\begin{prop}\label{partial.grading.k}\prlabel{partgradk} \cite{AAB}
There is a bijective correspondence between subgroups of $G$ and partial $G$-gradings of the base field $k$.
\end{prop}

\begin{proof} The partial grading corresponds to a linear map $\lambda : k^G \rightarrow k$ determined by the values $\lambda_g = \lambda(p_g)$. The support of the partial grading is the set $H = \{g \in G; \lambda_g \neq 0\}$. This set is a subgroup of $G$ and it can be shown that 
\begin{equation}\eqlabel{partactk}
\lambda_g  = \frac{1}{|H|} \delta_{gH,H} 
\end{equation}
for all $g \in G$. 
\end{proof}

Since $k \simeq End(k)$, it follows by \exref{partact} that the partial actions described in
\prref{partgradk} induce partial representations of $k^G$ on $k$ by the following formula (which is in fact the same as
\equref{partactk})
\[
\pi:k^G\to k,\ \pi(p_g)  = \frac{1}{|H|} \delta_{gH,H}
\]
for all $g \in G$.

Before we consider partial gradings on the matrix algebra $A=M_n(k)$ let us recall the following method to construct
usual $G$-gradings on a matrix algebra (see \cite{DINM} for more details). 
Given $(g_1,g_2, \ldots, g_n) \in G^n$, there is a $G$-grading on $M_n(k)$ given by the formula
\begin{equation} \label{formula.grading.DINM}
\deg(E_{i,j}) = g_ig_j^{-1}, \ 1 \leq i, j
\leq n
\end{equation}
A $G$-grading of $M_n(k)$ in which every elementary matrix $E_{i,j}$ is homogeneous is called a \emph{good grading} in \cite{DINM} (and an \emph{elementary grading} in \cite{BSZ}). 
Note that  the elements $(g_1,g_2, \ldots, g_n)$ and $(g_1k, g_2k, \ldots, g_nk)$ define the same grading for any $k \in G$. It is easy to see that there are no more repetitions if we fix $g_1=1$, and therefore we have a different $G$-grading of $M_n(k)$ for each element of $G^{n-1}$.  It turns out that every good $G$-grading is obtained in this way.  

\begin{prop} \cite[Prop.~2.1]{DINM}
There is a bijective correspondence between good $G$-gradings of $M_n(k)$ and elements of $G^{n-1}$.
\end{prop}

We should  mention that this correspondence is given using a (seemingly) different grading: given a good $G$-grading of $M_n(k)$, the idempotents $E_{i,i}$ are all homogeneous and hence have degree $1$, and the degrees of the other elementary matrices  are determined  by the elements $h_1,h_2, \ldots, h_{n-1}$ such that $\deg(E_{i,i+1}) = h_i$. In fact, it is easy to see that the other degrees are given by 
\begin{align}
\deg(E_{i,j}) & = h_ih_{i+1} \cdots h_{j-1}, \ \ & \deg(E_{j,i}) & = h_{j-1}^{-1}h_{j-2}^{-1} \cdots h_{i}^{-1} \label{formula.grading.DINM.2}
\end{align}
for $1 \leq i < j \leq n.$ Conversely, any $(n-1)$-tuple of elements of $G$ determines a good $G$-grading of $M_n(k)$ in this manner.

It is easy to see that if one begins with the $G$-grading of $M_n(k)$ defined by  $(g_1,g_2,\ldots,g_n)$ via \eqref{formula.grading.DINM}, then
the elements $h_i = g_ig_{i+1}^{-1}$ define the same grading via \eqref{formula.grading.DINM.2} . Conversely, given $(h_1,h_2, \ldots, h_{n-1})$, if we define $g_1 = 1$ and $g_i = h_{i}^{-1}h_{i-1}^{-1} \cdots h_1^{-1}$ for $2 \leq i \leq n$ then the former formulas for the grading defined by the $h_i$'s defines the same grading as the one obtained from the $g_i$'s by \eqref{formula.grading.DINM}.

By analogy with $G$-gradings of $M_n(k)$, we will say that a partial $G$-grading of $M_n(k)$ is \emph{good} if the elementary matrices $\{E_{i,j}; 1 \leq i,j \leq n\}$ are simultaneous eigenvectors for all operators $p_g \cdot \_ $.  In what follows we will describe all good partial $G$-gradings of $M_n(k)$ for a finite  group $G$.

Recall from \cite[Thm.~3.9]{AAB} that if $H$ is a cocommutative Hopf algebra and $A$, $B$ are two left partial
$H$-module algebras then $A\ot B$ is again a partial $H$-module algebra.

 As a consequence, if $G$ is abelian,
 $A=M_n(k)$ is endowed with a $G$-grading by a fixed $n$-tuple $(g_1,\ldots,g_n)$ as above; and $B$ is any partially
$G$-graded algebra then we obtain a partial $G$-grading on $B \otimes A$. In particular,
we may tensor $A=M_n(k)$ by the partial $G$-grading on $k$ of the proposition \ref{partial.grading.k} in order to obtain
a partial $G$-grading on $k \otimes A \simeq A$. This partial grading is given explicitly by  
\begin{equation}\label{equation.partialgrading.matrix.algebra}
 p_g \cdot  E_{i,j} 
= \frac{1}{|H|} \delta_{gH,g_ig_j^{-1}H} \ E_{i,j}.
\end{equation}
because
\begin{align*}
 p_g \cdot (1 \otimes E_{i,j}) & =  \sum_x (p_{gx^{-1}} \cdot 1) \otimes ( p_x \cdot E_{i,j}) = 
\frac{1}{|H|} \sum_{gx^{-1} \in H} 1 \otimes  \delta_{x,g_ig_j^{-1}} E_{i,j}  \\
&
= \frac{1}{|H|} \delta_{gH,g_ig_j^{-1}H} 1 \otimes E_{i,j}.
\end{align*}
Note that $p_g \cdot ( p_g \cdot E_{i,j}) = (1/|H|) p_g \cdot E_{i,j}$, and hence the operator $p_g \cdot \_$ is a projection 
if and only if $|H| = 1$; in particular, if $|H| \neq 1$ this grading is properly partial.

By \exref{partact}, this partial action induces a partial representation $\pi:k^G\to M_n(k)$ given by
\begin{equation}
\pi(p_g) = \frac{1}{|H|} \sum_i \delta_{g_iH,gH} E_{i,i}.
\end{equation}
As $M_n(k)\simeq \End_k(k^n)$, this means that $k^n$ is a partial $k^G$-module, where 
\[
p_g \cdot  e_i = \frac{1}{|H|}  \delta_{g_iH,gH} e_i. 
\]

Another way of comparing a partial grading with a usual grading is to take $A_g$ to be the image of the operator
$\pi(p_g)$; with respect to the last example, we get 
\begin{equation}
A_g = \bigoplus_{\substack{1 \leq i,j \leq n\\  gH = g_ig_j^{-1}H}}  kE_{i,j}
\end{equation}
and $A_g = A_t$ if and only if $gH = tH$. Something curious occurs if we 
consider the quotient group $G/H$: If for each coset $gH$ we define 
\begin{equation}
A_{gH} = A_g
\end{equation}
we obtain $A = \bigoplus_{gH \in G/H} \ A_{gH}$ and also  
\[
A_{gH} A_{tH} = A_g A_t \subseteq  A_{gt} = A_{gtH}
\]
i.e., if we ``quotient $H$ out'' we obtain a  $G/H$-grading of $A$. We will see in the following that indeed 
there is a correspondence between partial $G$-gradings and gradings by $G/H$. 

If $H,K$ are Hopf algebras and $\varphi: H \rightarrow K$ is a morphism of Hopf algebras, then the induction
functor  ${}_K \mathcal{M} \rightarrow {}_H \mathcal{M}$ 
(which associates to $M \in {}_K\mathcal{M}$ the $H$-module $M$ where $h \cdot_\varphi m = \varphi(h) m$) is
monoidal and therefore takes $K$-module algebras to $H$-module algebras. As a consequence of \prref{Hoidfunctor},
 the same holds for partial representations of $K$ and partial $K$-module algebras. 

In particular, if $G$ is a finite group and $H$ is a normal subgroup then the projection $\pi : G \rightarrow G/H$ induces a
morphism $\pi^*: k^{G/H} \rightarrow k^{G}$ of Hopf algebras, and therefore if $A$ is a $k^{G}$-module algebra then $A$ is canonically a $k^{G/H}$-module algebra by 

\begin{equation} \label{formula.pullback}
p_{gH} \acts a = \pi^*(p_{gH}) \rightarrow a = \sum_{h \in H} p_{gh} \rightarrow a. 
\end{equation}

The converse is not true, i.e., usually $G/H$-gradings cannot be lifted to $G$-gradings. But they can be canonically lifted to \emph{partial} $G$-gradings.

In what follows, to avoid confusion, we will write $p_g$ for the elements of the canonical basis of 
$k^G$ and $\q_{gH}$ (not $p_{gH})$ for the elements of the basis of $k^{G/H}$. 

\begin{prop} \prlabel{prop.normalsubgroup}
Let $G$ be a finite group, $H$ be  a normal subgroup of $G$, and let $A$ be a  $G/H$-graded algebra.  Then $A$ is also 
a partial $G$-graded algebra by 

\begin{equation} \label{eqn.up}
p_g \cdot a = \frac{1}{|H|}  \q_{gH} \rightarrow a .
\end{equation}

\end{prop}

\begin{proof}
It is clear that $\sum_{g} p_g \cdot a = a$ for all $a \in A$. With respect to the composition of $p_g$ and $p_t$, we have 
\beqnast
p_g \cdot ( p_t \cdot a) & = & 
\frac{1}{|H|^2} \q_{gH} \rightarrow (\q_{tH} \rightarrow a)  
 =  \delta_{tH,gH}\frac{1}{|H|^2} (\q_{tH} \rightarrow a ) \\ 
\eqnast
and 
\beqnast
 (p_{gt^{-1}} \cdot 1_A)(p_t \cdot a) 
 &= &   (p_t \cdot a) (p_{t^{-1}g} \cdot 1_A) =  \delta_{tH,gH} \frac{1}{|H|^2}(\q_{tH} \rightarrow a ).  
\eqnast
and hence (iii) holds. And applying $p_g$ on a product we obtain 
\beqnast
p_g \cdot ab & = & \frac{1}{|H|}   \q_{gH} \rightarrow ab   =  \frac{1}{|H|}  
 \sum_{kH \in G/H}(\q_{gk^{-1}H} \rightarrow a )(\q_{kH} \rightarrow b)   \\
& = & \frac{1}{|H|^2}   \sum_{k \in G}(\q_{gk^{-1}H} \rightarrow a )(\q_{kH} \rightarrow b)    =   \sum_{k \in G}(p_{gk^{-1}} \cdot a )(p_{k} \cdot b)   
\eqnast
thus showing that this is indeed a partial $G$-grading of $A$.
\end{proof}

Our next aim is to investigate which kinds of partial $G$-gradings on an algebra $A$ can be obtained from global $G/H$
gradings for a normal subgroup $H$ of $G$, as in the previous proposition. 

We say that a partial $G$-grading of $A$ is \emph{linear} if $1_A$ is an eigenvector for each operator $p_g \cdot \_ \ $, i.e., for each $g \in G$ there is a $\lambda_g \in k$ such that $p_g \cdot 1_A = \lambda_g 1_A$. 
In this case, the restriction of the partial grading to $k1_A$   corresponds to a partial grading of the field $k$ and these are listed
in Example \ref{partial.grading.k} : There is a subgroup $H$ of $G$ such that 
$\lambda_g = \delta_{gH,H} \frac{1}{|H|}$.
 We will call $H$ the \emph{linear support} of the partial grading of $A$.

Partial $G$-graded algebras with linear support $H$ form a full subcategory of $\ParAct_{k^G}$ which we will call
$\ParAct_{k^G,H}$. The last result provides a functor $\uparrow_{G/H}^G: \Alg({}_{k^{G/H}}\Mm) \rightarrow
\ParAct_{k^G,H}$, where $A \uparrow_{G/H}^G$ denotes the partial $G$-grading on $A$ obtained from its $G/H$-grading. 
It is clear from \eqref{eqn.up} that every morphism $\phi: A \vai B$ of $G/H$-graded algebras is also a morphism of partial $G$-graded algebras and therefore we define $ \phi\uparrow_{G/H}^G = \phi$.

There is a functor going in the opposite direction.

\begin{prop} 
Let $G$ be a finite group, $k$ be a field such that $\mathrm{char}(k) \nmid |G|$, and let $H$ be a normal subgroup of
$G$. Every linear partial $G$-grading of $A$ with linear support $H$ determines a $G/H$-grading of
$A$ by the formula 
\begin{equation} \label{eqn.down}
q_{gH} \acts a =  |H| \ p_g \cdot a.
\end{equation}
Moreover, this formula defines a functor $\downarrow_{G/H}^G: \ParAct_{k^G,H} \rightarrow \Alg({}_{k^{G/H}}\Mm)$.
\end{prop}
\begin{proof}
Consider a linear partial grading of $A$ with support $H$. 
Note that 
\begin{equation*}
p_g \cdot a = p_g \cdot (1_A \ a) = \sum_{t}(p_{gt^{-1}} \cdot 1_A)(p_t \cdot a) =
 \sum_{t} \delta_{gH,tH} \dfrac{1}{|H|} ( p_t \cdot a).
\end{equation*}
and therefore
\begin{equation}
p_g \cdot a =
  \dfrac{1}{|H|}  \sum_{t \in gH} (p_t \cdot a). 
\end{equation}
It follows that $p_g \cdot a = p_r \cdot a$ if $gH = rH$.

Now consider the partial $k^{G/H}$-grading of $A$ defined by the induction functor associated to the canonical morphism of Hopf algebras 
$\pi^*:k^{G/H} \rightarrow k^G$.  Then 
\begin{equation*}
q_{gH} \cdot a = (\pi^*(q_{gH})) \cdot a = \sum_{t \in gH} p_{t} \cdot a = |H| \ p_g \cdot a.
\end{equation*} 
Applying $p_g $ to $1_A$ we get
\begin{equation*}
q_{gH} \cdot 1_A = |H| \ p_g \cdot 1_A
=  \delta_{gH,H} 1_A = \varepsilon(q_{gH}) 1_A
\end{equation*}
and therefore the $k^{G/H}$-action is \emph{global}, i.e., this is a $G/H$-grading of $A$.
\end{proof}

\begin{thm} \thlabel{thm.linear.grading}
Let $G$ be a finite group, $k$ be a field such that $\mathrm{char}(k) \nmid |G|$, and let $H$ be a normal subgroup of
$G$. The functor $ \downarrow_{G/H}^G: \ParAct_{k^G,H} \rightarrow \Alg({}_{k^{G/H}}\Mm)$ is an isomorphism of
categories, with inverse 
$\uparrow_{G/H}^G: \Alg({}_{k^{G/H}}\Mm) \rightarrow \ParAct_{k^G,H}$.
\end{thm}

\begin{proof}
Let $A$ be a $G/H$-graded algebra. Calculating the $G/H$-grading defined by $(A \downarrow_{G/H}^G)\uparrow^G_{G/H}$ we obtain 
\begin{equation*}
q_{gH} \acts a = \sum_{t \in gH} p_t \bullet a = \sum_{t \in gH} \frac{1}{|H|} q_{tH} \rightarrow a =  q_{gH} \rightarrow a
\end{equation*}
Conversely, beginning with a partial $G$-grading of $A$, the partial $G$-graded algebra $(A \uparrow^G_{G/H}) \downarrow_{G/H}^G$ is 
\begin{equation*}
p_{g} \bullet a = \frac{1}{|H|} q_{gH} \acts a = \frac{1}{|H|} |H| \ p_g \cdot a = p_g \cdot a.
\end{equation*}
Since these functors behave as identities on morphisms, this shows that they are mutually inverses.
\end{proof}

From this last result we obtain the description of all good partial $G$-grading of $M_n(k)$. This description can be obtained from \cite[Thm.~3.6]{AAB} if one translates partial gradings of $M_n(k)$ in terms of partial gradings of an associated $k$-linear category, but we provide a self-contained proof below. 
 
\begin{cor}
If $G$ is a finite abelian group and $k$ is a field such that $\mathrm{char}(k) \nmid |G|$ then 
every good partial $G$-grading of $M_n(k)$ is a linear grading. Moreover, given a  subgroup $H$ of $G$, there is a bijective correspondence between  good partial $G$-gradings of $M_n(k)$ with linear support $H$ and good  $G/H$-gradings of $M_n(k)$. 
\end{cor}
\begin{proof} 
We begin by proving that every good partial $G$-grading is linear.
By hypothesis, for each $i,j \in \{1, 2, \ldots, n\}$ there is a family  of scalars $\lambda_g^{i,j}$ indexed by $G$ such that
 $p_g \cdot E_{i,j} = \lambda_g^{i,j} E_{i,j}$. 
 
Fix $i < j $. Since $\sum_g p_g \cdot a = a$ for every $a \in M_n(k)$ there is at least one element $s \in  G$ such that $p_s \cdot E_{i,j} \neq 0$. It follows from the first equality of the third item of \deref{def.partial.grading} that  
\begin{equation} \label{pr.ps}
p_r \cdot( p_s \cdot E_{i,j} )= (p_{rs^{-1}}\cdot (\sum_{k} E_{k,k})) \lambda_s^{i,j} E_{i,j} = \lambda_{rs^{-1}}^{i,i}\lambda_s^{i,j} E_{i,j}.
\end{equation}
and from the second equality, plus the fact that $G$ is abelian, we obtain $p_r \cdot( p_s \cdot  E_{i,j} ) = \lambda_{rs^{-1}}^{j,j}\lambda_s^{i,j} E_{i,j}$, and therefore $\lambda_{g}^{j,j} =\lambda_{g}^{i,i}$ for all $g \in G$. 

On the other hand, for each $i \in \{1, \ldots, n\}$ the subalgebra $k E_{i,i}$ is a partial $G$-graded subalgebra of $M_n(k)$ (the only part that requires some care  is to check that the same item (3) of \deref{def.partial.grading} holds). Therefore, not only  the partial grading of $k E_{i,i}$  is defined by a subgroup $H$ of $G$ as in \prref{partgradk}, but also the same subgroup $H$ appears for every index $i$; in particular, if follows that the partial grading of $M_n(k)$ is linear with support $H$. 
 \thref{thm.linear.grading} then says  that this partial grading corresponds to a $G/H$-grading of $M_n(k)$ via equation \eqref{eqn.down}.

Let us check now that  this is a good $G/H$-grading.
Let $i,j \in \{1, \ldots ,n\}$, $i \neq j$, and let $S_{i,j} = \{ t \in G; p_t \cdot E_{i,j} \neq 0\}$. If $t$ and $s$ are in $S_{i,j}$ then 
\[
0 \neq p_t \cdot  (p_s \cdot E_{i,j}) = (p_{ts^{-1}} \cdot I)(p_s \cdot E_{i,j})
\]
and therefore $ts^{-1} \in H$, i.e., $tH = sH$. Conversely, it can be shown that if $tH = sH$ and $p_t \cdot E_{i,j} \neq 0$ then 
$p_s \cdot E_{i,j} \neq 0$. Therefore there exists $t = t_{i,j} \in G$ such that $S_{i,j} = t_{i,j}H$.

Consider $r,s \in  t_{i,j}H$. Then $p_r \cdot ( p_s \cdot E_{i,j}) = \lambda_s^{i,j} p_r \cdot E_{i,j} = \lambda_s^{i,j} \lambda_r^{i,j} E_{i,j}$, with both $\lambda_s^{i,j}$ and $ \lambda_r^{i,j}$ nonzero, but from Equation \eqref{pr.ps} it also follows that
\[
p_r \cdot ( p_s \cdot E_{i,j}) = \lambda_{rs^{-1}}^{i,i}\lambda_s^{i,j} E_{i,j} = 
\frac{1}{|H|}\lambda_s^{i,j} E_{i,j}
\]
and hence $\lambda_r^{i,j} = 1/|H|$ if $r \in t_{i,j}H$ (and zero otherwise). 

Using the formula for the $G/H$-grading on $M_n(k)$ provided by  \eqref{eqn.down} we obtain
\[
q_{gH} \rightarrow E_{i,j} = |H| p_g \cdot E_{i,j} = \delta_{t_{i,j}H, gH} E_{i,j}
\]
proving that every elementary matrix is homogeneous. Therefore the  $G/H$-grading associated to a partial good grading is also a good grading.

Conversely, it is clear that the bijective correspondence of \thref{thm.linear.grading} sends  good $G/H$-gradings of $M_n(k)$ to  good partial $G$-grading of $M_n(k)$ with linear support $H$.
\end{proof}

\subsection{Partial $\mathbb{Z}_2$ gradings}

We recall that for $H =  (kC_2)^*$ the map $[p_g] \in H_{par} \mapsto x \in  k[x]/\langle x(2x-1)(x-1) \rangle$ is an isomorphism of $k$-algebras. The Chinese remainder theorem says that  this last one  is also isomorphic as an algebra to 
\[
\dfrac{k[x]}{\langle x \rangle} \times \dfrac{ k[x]}{\langle x-1 \rangle} 
\times \dfrac{ k [x]}{ \langle 2x-1 \rangle}
\]
 In particular, if $\pi: H \rightarrow \End(V)$ is a partial representation then $V$ is an $H_{par}$-module and therefore it breaks up as a direct sum $V = V_0 \oplus V_1 \oplus V_{1/2}$, where $V_k$ is the Eigenspace of $\pi(p_g)$ with respect to the eigenvalue $k$.  Identifying $\pi$ with its extension to $H_{par}$, the projections onto each subspace are given by 
$P_0 = \pi([p_e][p_e] - [p_e][p_g])$, 
$P_1 = \pi([p_g][p_g] - [p_e][p_g])$ and 
$P_{1/2} = 4[p_e][p_g]$ respectively.

It can be proved that $P_j^ 2 = P_j$ for $j=0,1,1/2$, that $P_j P_k = 0 $  if $j \neq k$, $P_0 + P_1 + P_{1/2} = I_V$ and that 
\begin{align*}
 \pi([p_e] )P_0 & = P_0    & \pi([p_g])P_0 & = 0     &  \pi([p_e ])P_{1/2} & =  (1/2) P_{1/2} \\
 \pi([p_g]) P_1 & = 0        & \pi([p_g])P_1 & = P_1  & \pi([p_g]) P_{1/2} &= (1/2) P_{1/2}& 
\end{align*}
since these equations already hold in $H_{par}$. We mention that the expressions for the projections can also be
obtained via the Chinese remainder theorem. 

From the correspondence between partial representations of  $H$ on $B$ and algebra morphisms from $H_{par}$ to $B$, we conclude, in particular,  that an operator $T \in \End(V)$ defines a partial representation of $ (kC_2)^*$ on  $V$  by $\pi(p_g) = T$
if and only if its minimal polynomial is a divisor of $p(x) = x(2x-1)(x-1) $, and it will be diagonalizable (note that a (global) representation of $ (kC_2)^*$ is a (global) $\Z_2$-grading of $V$ and these are given by algebra morphisms from $k[x]/\langle x(x-1) \rangle$ to $\End(V)$).

We recall that one of the motivations for studying partial representations is that any partial action of a Hopf algebra $H$ on a algebra $B$ defines a partial representation $\pi: H \rightarrow B$ by $\pi(h) (b) = h \cdot b$. A beautiful application of this procedure is the analysis of partial $\Z_2$-gradings of algebras, (i.e., partial actions of $(kC_2)^*$)  via the corresponding partial representations. 

\begin{thm} \label{lemma.decomposition.algebra}
Let $H = (kC_2)^ *$.
Let $A$ be a partially $\Z_2$-graded algebra and let $A = A_0 \oplus A_1 \oplus A_{1/2}$ be its decomposition as an $H_{par}$-module.
\begin{enumerate}
\item $ B = A_0 \oplus A_1$ is a subalgebra of $A$ which is $\Z_2$-graded by  $B_{\overline{0}} = A_0$ and $B_{\overline{1}} = A_1$.
\item  $A_{1/2}$ is an ideal of $A$
\item The unit of $A$ decomposes as  $1_A = u + v$ where  $u,v$ are orthogonal idempotents, $u \in A_0$ and $v \in A_{1/2}$.
\item $uAv = vAu=0$ and therefore $A \simeq B \times A_{1/2}$ as an algebra.  
\end{enumerate}
\end{thm}

\begin{proof} The first item can be checked on a case-by-case basis. For instance, let $a,b \in A_{1}$.
Since $p_e \cdot x = (1 - p_g) \cdot x $, for all $x \in A$, it follows that  
$p_e \cdot a = p_e \cdot b = 0$ and hence
\begin{equation*}
p_g \cdot (ab) = (p_g \cdot a)(p_e \cdot b) + (p_e \cdot a)(p_g \cdot b) = 0
\end{equation*}
 which implies that   $ab \in A_0$. The other cases are analogous.  The same kind of computations also show that $A_{1/2}$ is an ideal.

For the third item, we can write $1 = u_0 + u_1 + v$ with $u_0 \in A_0, u_1 \in A_1, v \in A_{1/2}$. Given $a \in B = A_0 \oplus A_1$, we have $ a = a(u_0 + u_1) + av$ with $a(u_0 + u_1) \in B$ and   $av \in A_{1/2} $, which implies that $a = a(u_0 + u_1)$ and $av = 0$ for all $a \in B$. Of course we also have 
$a = (u_0 + u_1)$a and $va = 0$. Therefore $u  = u_0 + u_1$ is the unit of $B$ and, since $B = A_0 \oplus A_1$ is a $\Z_2$-grading of $B$ one concludes that $u = u_0$.  In particular, $uv = vu = 0$ and $1_A = u + v$.

Consider now an element $uav$, $a \in A$. Let us show that $uav = 0$.  Applying $p_g$ twice to 
the element $uav $ we obtain
\[
p_g \cdot (p_g \cdot (uav )) = (1/4) uav.
\]
On the other hand, 
\[
p_e \cdot 1_A = p_g \cdot 1_A = u + (1/2)v,
\]
and we also have
\begin{equation*}
 p_g \cdot (p_g \cdot uav )  = (p_e \cdot 1_A)(p_g \cdot uav ) 
=  (u + (1/2)v)( (1/2) uav) 
=  (1/2) uav 
\end{equation*}
which implies that $uav = 0$ for all $a \in A$. In a similar way we can use the equality 
\[
p_g \cdot (p_g \cdot (vau)) = (p_g \cdot vau)(p_e \cdot 1_A)
\]
to prove that $vAu$ is also the zero subspace. 

From these results we conclude that $u$ and $v$ are two orthogonal central idempotents which add up to $1_A$, and since $uAu = A_0 \oplus A_1$  and $vAv = A_{1/2}$ we deduce that $A$ is isomorphic to $(A_0 \oplus A_1) \times A_{1/2}$ as an algebra.
\end{proof}

This result also suggests a simple construction of examples of partial $\Z_2$-gradings: given a $\Z_2$-graded algebra $B$ and another $k$-algebra $C$, the algebra $B \times C$ has a partial grading given by 
\begin{align*}
p_g \cdot (b,c) = (p_g \rhd b, c/2).
\end{align*}
The previous theorem says that every partially $\Z_2$-graded algebra can be  obtained in this manner.

\section*{Acknowledgments}

One of the authors, E.\ Batista, is supported by CNPq (Ci\^{e}ncia sem Fronteiras), proc n\b{o} 23644/2012-8, and he
would like to thanks the D\'epartement des Math\'ematiques de l'Universit\'e Libre de Bruxelles for their kind
hospitality. M.M.S. Alves is partially supported by CNPq, project n. 304705/2010-1. M.M.S.\ Alves and E.\ Batista are
also partially supported by CNPq, project n. 477880/2012-6 and Funda\c{c}\~ao Arauc\'aria, project n. 490/16032.
The last author would like to thank the FNRS for the CDR ``Symmetries of non-compact non-commutative spaces: Coactions
of generalized Hopf algebras.'' that partially supported this collaboration.

\end{document}